\documentclass[12pt]{article}
\usepackage{setspace}

\usepackage{amsthm}
\usepackage{amsfonts}
\usepackage{amsmath}
\usepackage{mathrsfs}
\usepackage{dsfont}
\usepackage{graphicx}
\usepackage{xcolor}
\usepackage[ruled,vlined]{algorithm2e}
\usepackage{amssymb}
\usepackage[english]{babel}
\usepackage[margin=1in]{geometry}
\usepackage{float}
\usepackage{natbib}

\usepackage{booktabs}
\usepackage{comment}

\setcitestyle{authoryear,open={(},close={)}}
\usepackage{subcaption}
\usepackage{multirow}
\usepackage{mathtools}
\usepackage{authblk}

\newtheorem{theorem}{Theorem}[section]

\newtheorem{assumption}[theorem]{Assumption}

\newtheorem{proposition}[theorem]{Proposition}
\newtheorem{observation}[theorem]{Observation}

\theoremstyle{definition}
\newtheorem{definition}{Definition}[section]
\theoremstyle{definition}
\newtheorem{example}{Example}[section]
\theoremstyle{definition}

\theoremstyle{definition}

\title{Robust Data-Driven Quasiconcave Optimization}

\author[1]{Jian Wu}
\author[2]{William B. Haskell}
\author[3]{Wenjie Huang}
\author[4]{Huifu Xu}
\affil[1]{\footnotesize Department of Decision Analytics and Operations, City University of Hong Kong, Kowloon, Hong Kong, jwu424@cityu.edu.hk}
\affil[2]{\footnotesize Daniels' School of Business, Purdue University, West Lafayette, Indiana 47907, whaskell@purdue.edu}
\affil[3]{\footnotesize Department of Data and Systems Engineering \& Musketeers Foundation Institute of Data Science, The University of Hong Kong, Pokfulam, Hong Kong, huangwj@hku.hk}
\affil[4]{\footnotesize Department of Systems Engineering and Engineering Management, The Chinese University of Hong Kong,
Shatin, N.T., Hong Kong, hfxu@se.cuhk.edu.hk}
\date{}

\begin{document}



\maketitle

\begin{abstract}
We investigate a data-driven
quasiconcave maximization problem where information about
the objective function is limited to a finite sample of data points.
We begin by defining an ambiguity set for admissible objective functions based on available partial information about the objective. This ambiguity set consists of those 
quasiconcave functions that majorize a given data sample, and that satisfy additional functional properties (monotonicity, Lipschitz continuity, and permutation invariance).
We then formulate a robust optimization (RO) problem which maximizes the worst-case objective function over this ambiguity set.
Based on the quasiconcave structure in this problem, we explicitly construct the upper level sets of the worst-case objective at all levels. 
We can then solve the resulting RO problem efficiently by 
doing binary search over the upper level sets and solving a logarithmic number of convex feasibility problems. 
This numerical approach differs from traditional subgradient descent and support function based methods for this problem class. 
While these methods can be applied in our setting, 
the binary search method displays superb finite 
convergence to the global optimum, whereas the others do not. This is primarily because binary search fully exploits the specific structure of 
the worst-case quasiconcave objective, which leads to an explicit and general convergence rate in terms of the number of convex optimization problems to be solved.
Our numerical experiments on a Cobb-Douglas production efficiency problem {\color{black} and a fair resource allocation problem} demonstrate the tractability of our approach.
\end{abstract}

\textbf{Keywords:} robust optimization, quasiconcave functions, upper level sets, binary search

\section{Introduction}

We often face optimization problems where there is uncertainty about the problem data and the objective function is ambiguous.
In this paper, we examine quasiconcave maximization problems where there is epistemic uncertainty caused by a lack of information about the true objective function.
This is in contrast to aleatoric (stochastic) uncertainty, caused by known sources of random variability. We develop \textcolor{black}{a robust optimization (RO)} approach for such quasiconcave maximization problem which maximizes the worst-case valuation over an ambiguity set of admissible objective functions. Although quasiconcavity is a substantial generalization of concavity, our approach still permits efficient numerical solution, by exploiting the specific structure of the problem.

We have a convex and compact set of feasible decisions 
${\cal Z} \subset \mathbb{R}^T$ for $T \geq 1$ and a convex and compact range of outcomes (outputs) ${\cal X} \subset \mathbb{R}^N$ for $N \geq 1$. {\color{black}We equip $\mathbb{R}^N$ with the component-wise partial order $\geq$, where $x \geq x'$ for $x, x' \in \mathbb{R}^N$ when $x_n \geq x_n'$ for all components $n=1,\ldots,N$.}
There is a vector-valued mapping $G : {\cal Z} \rightarrow {\cal X}$ which represents $N$ performance measures that depend on our decisions $z \in {\cal Z}$, where $G(z) = (g_1(z), \ldots, g_N(z))$ and where each $g_n : {\cal Z} \rightarrow \mathbb{R}$ is the component $n$ function. {\color{black}We suppose that $G$ is vector-valued concave, so for any $z,z'\in {\cal Z}$ we have $G(\lambda z+(1-\lambda) z') \geq 
\lambda G(z) +(1-\lambda) G(z')$ for all $\lambda \in (0,1)$.} For instance, $G(z)$ may be the production output for a set of different goods as a function of the material inputs. Alternatively, $G(z)$ could be the vector of rewards over a set of $N$ uncertain scenarios.

We interpret a quasiconcave $f : {\cal X} \rightarrow \mathbb{R}$ as a valuation function for $G(z)$ where larger values are preferred, so if $f$ 
was fully specified we would solve:
\begin{equation}
\label{prob:intro}
\max_{z \in {\cal Z}} f(G(z)).
\end{equation}
In our setting, $f$ is ambiguous so the valuation for $G(z)$ is not precisely known.
For instance, the objective function $f$ \textcolor{black}{in the multi-objective optimization problem~\eqref{prob:intro}} may be ambiguous due to incomplete information about the DM's preferences.
We only have an ambiguity set ${\cal F}$ of possible quasiconcave valuation functions $f$.
The resulting RO problem is:
\begin{equation}
\label{prob:robust_intro}
\max_{z \in {\cal Z}} \left\{\psi_{{\cal F}}(x) \triangleq \min_{f \in {\cal F}} f(G(z))\right\}.
\end{equation}
The objective $\psi_{{\cal F}}$ of Problem~\eqref{prob:robust_intro} is quasiconcave, since quasiconcavity is preserved under minimization.
We additionally suppose that ${\cal F}$ consists of $L-$Lipschitz continuous functions, and that $f(\hat{x}) > -\infty$ for all $f \in {\cal F}$ for some $\hat{x} \in {\cal X}$. In this case, $\psi_{{\cal F}}(x) > -\infty$ for all $x \in {\cal X}$ so Problem~\eqref{prob:robust_intro} is well-defined.

Problem~\eqref{prob:robust_intro} generalizes several existing RO models for convex optimization problems.
For instance, it recovers the cases where (i) $f$ is linear and ${\cal F}$ is a polyhedron; and (ii) $f$ is convex and ${\cal F}$ is determined by constraints on function values.
Many of the usual RO numerical methods depend on convex duality for tractable reformulation of the overall problem (as in robust linear programming). In these methods, we take the dual of the inner minimization problem, establish strong duality, and then obtain a single overall maximization problem. However, these methods cannot be applied here because the inner minimization in Problem~\eqref{prob:robust_intro} is a non-convex optimization problem (as the property of quasiconcavity is not preserved by convex combination). We have to develop an entirely new numerical approach to solve Problem~\eqref{prob:robust_intro}.

The properties of quasiconcavity/quasiconvexity have broad practical relevance. For instance, quasiconcavity/quasiconvexity are important in computer vision \citep{ke2006uncertainty,ke2007quasiconvex}, optimal control \citep{ning2019time}, and production economics \citep{bradley1974fractional,mukherjee_least_2024}.
Quasiconcavity is also the most general form of diversification-favoring behavior (for preferences over random rewards) in decision theory.
One of the most prevalent examples of a quasiconcave choice function is the certainty equivalent (see, e.g., \cite{Ben-Tal2007}). The indices of acceptability proposed by \cite{cherny2009new} are also quasiconcave.
\cite{brown2009satisficing} develop the class of satisficing measures, which are quasiconcave and are based on the idea of meeting a target (satisficing rather than optimizing).
\cite{frittelli2014risk} develop a theory of quasiconcave evaluation functions over distributions on $\mathbb{R}$, to represent preferences over lotteries.

Quasiconvexity is also prominent in finance (for random losses). \cite{mastrogiacomo2015portfolio} study portfolio optimization with quasiconvex risk measures, which are a weaker expression of diversification-favoring preferences compared to convex risk measures. They obtain the equivalent saddle-point formulation of this problem, and identify necessary and sufficient optimality conditions for an optimal portfolio.
\cite{brown2012aspirational} develop the class of aspirational measures for random rewards, which are quasiconcave on one domain and quasiconvex on another.

The theory of quasiconvex optimization has been thoroughly studied.
\cite{luenberger1968quasi} develops Lagrange multiplier conditions for the global solution of minimizing a quasiconvex function over a convex set.
This result is generalized to quasiconvex minimization over Banach spaces in \cite{penot2004surrogate}.
\cite{agrawal2020disciplined} present composition rules for quasiconvex functions, and develop the framework of `disciplined quasiconvex programming' (which is analogous to disciplined convex programming, see, e.g., \cite{grant2006disciplined}) from a base set of canonical quasiconvex/quasiconcave functions. They also develop a numerical implementation for this class of problems.

Several numerical methods for solving quasiconvex programming problems have been developed, several of which are first-order methods. 
\cite{plastria1985lower} develops a cutting plane method for quasiconvex minimization based on the lower subdifferential.
\cite{kiwiel2001convergence} studies minimization of quasiconvex functions on Hilbert spaces via subgradient descent methods. They identify conditions for the asymptotic convergence to the minimum value and to the set of optimal solutions. Under stronger conditions, they find stepsizes that yield $\epsilon-$optimal solutions. \cite{xu2001level} develops the level function method for quasiconvex minimization. This procedure is based on iteratively constructing a sequence of level functions to the objective, and then solving a surrogate problem which minimizes the level function. They show that this procedure converges asymptotically to an optimal solution. 
\cite{konnov2003convergence} also studies subgradient methods for quasiconvex minimization on Euclidean spaces, and obtains convergence rates to the optimal solution.
\cite{hu2015inexact} consider inexact subgradient methods for quasiconvex minimization, and establish both asymptotic and finite convergence to near-optimal solutions. \cite{hazan2015beyond} develop a stochastic gradient descent method for quasiconvex minimization. \cite{yu2019abstract} propose a new subgradient method for quasiconvex minimization based on perturbations of each successive search direction, and they demonstrate its convergence within a general framework. 
\cite{grad2023relaxed} develop an inertial proximal point method for minimizing strongly quasiconvex functions, and establish asymptotic convergence of this procedure to an optimal solution.

Alternatively, it is natural to do binary search over the values of the objective function to solve quasiconvex programming problems.
This method requires solving a sequence of convex feasibility problems (see \cite{Boyd2004}). In particular, this is the method we develop in the present paper, where we search over the upper level sets of the worst-case objective of Problem~\eqref{prob:robust_intro}.

There is a vast literature on RO, and RO has been successfully applied to uncertain convex optimization problems (see, e.g., \cite{ben2002robust}).
However, robust quasiconcave maximization problems have not yet been extensively studied.
\cite{haskell2022preference} is the closest related work in this regard, it does preference robust optimization over mixtures of finitely many lotteries. Our present paper is for general robust quasiconcave maximization and applies to a much broader class of problems. The numerical method in our present paper is also much more efficient than the approach in \cite{haskell2022preference} for their particular setting.

Next we outline the main contributions of our present paper:
\begin{itemize}
\item (Modeling) We propose a new class of robust quasiconcave maximization problems where there is ambiguity (epistemic uncertainty) 
over the objective function due to lack of complete information.
This class of RO problems generalizes many robust convex optimization problems, and Problem~\eqref{prob:robust_intro} contains some well-studied RO problems as special cases.
We put special emphasis on the multivariate version of this problem, since there is often ambiguity about the cross-effects of different outputs in terms of the total valuation of $G(z)$. We construct a specific ambiguity set for the uncertain quasiconcave objective based on a data sample of points and key functional properties (monotonicity, Lipschitz continuity, and permutation invariance).

\item (Theory) We explicitly construct the upper level sets of the worst-case objective function, which reveals their polyhedral structure and their dependence on the problem data. We also connect with the theory of aspirational measures (see \cite{brown2012aspirational}) and show how to represent the worst-case objective in terms of a set of convex objective functions and targets. Conversely, we show how to construct a quasiconcave objective function from a given a set of convex functions and a family of targets.

\item (Computation) 
We propose a novel binary search  method for solving 
the robust quasiconcave maximization problem with a finite number
iterations. Unlike
traditional subgradient based methods (e.g.~descent direction method, cutting plane method, and level function method), binary search fully exploits the specific and explicit structure of the worst-case quasiconcave objective function and searches over the upper level sets of the function by solving a convex feasibility problem for each level.
As such, it requires to solve a logarithmic number of convex optimization problems rather than an infinite number of problems in order to obtain an exact optimal solution to Problem~\eqref{prob:robust_intro}. 
Our numerical tests for a Cobb-Douglas production efficiency problem {\color{black} and a fair resource allocation problem} show the effectiveness of our approach.
This approach is very efficient with respect to the amount of partial information (in terms of the number of data points) as well as the problem dimension.
\end{itemize}
%

The rest of the paper is organized as follows.
In Section~\ref{sec:preliminaries}, we review some technical preliminaries.
Section~\ref{sec:robust} then develops a quasiconcave RO model and gives the details of the specific ambiguity set we use in this paper.
In Section~\ref{sec:upper}, we characterize the upper level sets of the worst-case objective function.
Section~\ref{sec:binary} presents a binary search algorithm and gives its convergence properties.
Section~\ref{sec:numerical} reports numerical experiments, and the paper concludes in Section~\ref{sec:conclusion}.
{\color{black} In Appendix~\ref{sec:alternative}, we develop an alternative representation result for our RO model.} In Appendix~\ref{sec:invariance}, we show how to incorporate the additional property of permutation invariance into our RO model.
All proofs are gathered in the Appendix.

\paragraph{Notation}
Let $\mathbb{R}^N$ be the set of $N-$dimensional Euclidean vectors, and let $\|\cdot\|_p$ be the $p-$norm on $\mathbb{R}^N$ for $1 \leq p \leq \infty$.
We use $\mathbb{R}_{\geq 0}^N$ and $\mathbb{R}_{> 0}^N$ to denote the set of vectors in $\mathbb{R}^N$ with all non-negative components and all strictly positive components, respectively (for $N = 1$ we just write $\mathbb{R}_{\geq 0}$ and $\mathbb{R}_{>0}$).
Let $\mathbb N_{\geq 1}$ denote the positive integers.
For any $I \in \mathbb N_{\geq 1}$, we let $[I] \triangleq \{1, 2, \ldots, I\}$ denote the running index.
We let $\textsf{val}(\cdot)$ denote the optimal value of an optimization problem.

\section{Preliminaries}
\label{sec:preliminaries}

We repeat the following technical assumptions on the problem ingredients ${\cal Z}$, ${\cal X}$, and $G$.

\begin{assumption}
\label{assu:convex}
(a) Both ${\cal Z}$ and ${\cal X}$ are convex and compact.

(b) $G$ is vector-valued concave, i.e., each $g_n : {\cal Z} \rightarrow \mathbb{R}$ is concave for all $n \in [N]$.
\end{assumption}

We work in the set of bounded and measurable functions $f : {\cal X} \rightarrow \mathbb{R}$, denoted $\mathfrak{F}$, equipped with the supremum norm $\|f\|_{\infty} \triangleq \sup_{x \in {\cal X}} |f(x)|$. 
We recall the definition of the following key functional properties for our model.

\begin{definition}
Let $f : {\cal X} \rightarrow \mathbb{R}$ be a function.

(a) $f$ is concave if $f(\lambda x_1 + (1-\lambda) x_2) \geq \lambda f(x_1) + (1-\lambda) f(x_2)$, for all $x_1, x_2 \in {\cal X}$ and $\lambda \in [0,1]$.

(b) $f$ is quasiconcave if $f(\lambda x_1 + (1-\lambda) x_2) \geq \min\{f(x_1), f(x_2)\}$, for all $x_1, x_2 \in {\cal X}$ and $\lambda \in [0,1]$.

(c) $f$ is monotone (non-decreasing) if $f(x_1) \leq f(x_2)$ for all $x_1, x_2 \in {\cal X}$ with $x_1 \leq x_2$, where the inequality is interpreted component-wisely.

(d) $f$ is $L-$Lipschitz continuous (with respect to the infinity norm) if $|f(x_1) - f(x_2)| \leq L \|x_1 - x_2\|_{\infty}$ for all $x_1, x_2 \in {\cal X}$. 
\end{definition}
\noindent
The definition of Lipschitz continuity is with respect to the $\infty-$norm, which corresponds to a $1-$norm constraint on the magnitude of the subgradient.

Under Assumption~\ref{assu:convex}, without loss of generality we can restrict $\mathfrak{F}$ to consist of functions $f : {\cal X} \rightarrow [a,b]$ for a compact interval $[a, b]$. Then, $\mathfrak{F}$ is uniformly bounded.
Suppose ${\cal F} \subset \mathfrak{F}$ is additionally $L-$Lipschitz, then the functions in ${\cal F}$ are also equi-continuous. By the Arzel\'a-Ascoli theorem,  ${\cal F}$ is relatively compact under the norm topology. Moreover, since ${\cal F}$ is closed, then it is compact. 
Let $\mathcal{F}_{\text{Co}} \subset \mathfrak{F}$ be the set of all monotone and concave functions $f : {\cal X} \rightarrow \mathbb{R}$.
Let $\mathcal{F}_{\text{QCo}} \subset \mathfrak{F}$ be the set of all monotone and quasiconcave functions $f : {\cal X} \rightarrow \mathbb{R}$.
Let ${\cal F}_{\text{Lip}}(L) \subset \mathfrak{F}$ be the set of $L-$Lipschitz continuous functions.

Next we recall the definition of an affine majorant and a subgradient of functions in $\mathcal{F}_{\text{Co}}$, and a kinked majorant and an upper subgradient of functions in $\mathcal{F}_{\text{QCo}}$.

{\color{black}
\begin{definition}
(a) Let $f \in \mathcal{F}_{\text{Co}}$, $x \in {\cal X}$, and $\xi \in \mathbb{R}^{N}$. Then $h : {\cal X} \rightarrow \mathbb{R}$ defined by $h(y) \triangleq f(x) + \langle \xi,\,y - x \rangle$ for all $y \in {\cal X}$ is an {\em affine majorant} of $f$ at $x$ if  $h(y) \geq f(y)$ for all $y \in {\cal X}$.
In that case, $\xi$ is said to be a {\em subgradient} of $f$ at $x$. 
The set of all subgradients of $f$ at $x$ is called the {\em subdifferential} and is denoted by $\partial f(x)$.

(b) Let $f \in \mathcal{F}_{\text{QCo}}$, $x \in {\cal X}$, and $\xi \in \mathbb{R}^{N}$. Then $h : {\cal X} \rightarrow \mathbb{R}$ defined by $h(y) \triangleq f(x) + \max\{ \langle \xi,\,y-x \rangle,\,0\}$ for all $y \in {\cal X}$ is a {\em kinked majorant} of $f$ at $x$ if  $h(y) \geq f(y)$ for all $y \in {\cal X}$.
In that case, $\xi$ is said to be an {\em upper subgradient} of $f$ at $x$. 
The set of all upper subgradients of $f$ at $x$ is called the {\em upper subdifferential}  and is denoted by $\partial^+ f(x)$, see \cite{plastria1985lower}.
\end{definition}
}
\noindent
Any kinked majorant has convex upper level sets and is automatically in $\mathcal{F}_{\text{QCo}}$. We have the following results which characterize $L-$Lipschitz quasiconcave functions in terms of kinked majorants. Let $\|\cdot\|_*$ be the dual norm to $\|\cdot\|$ on $\mathbb{R}^N$.
We recall the following result on the characterization of quasiconcave functions from \cite{haskell2022preference}.

\begin{theorem}
\label{thm:char-qco-function}
Let $f \in \mathfrak{F}$. The following assertions hold.

(i) Suppose that $f$ has a kinked majorant everywhere in its domain.
Then, $f$ is quasiconcave, upper semi-continuous, and has a representation
\begin{equation}
f\left(x\right)=\inf_{j\in\mathcal{J}}h_{j}\left(x\right),\,\forall x\in\text{dom}\,f,\label{eq:Support-quasiconcave}
\end{equation}
where $\mathcal{J}$ is a (possibly infinite) index set and $h_{j}\left(x\right)=\max\left\{ \langle a_{j},\,x - x_{j}\rangle,\,0\right\} + b_j$ for all $j \in \mathcal{J}$ with constants $a_j \in \mathbb{R}^N$, $x_j \in \mathbb{R}^N$, and $b_j \in \mathbb{R}$.

(ii) Suppose $f$ is quasiconcave and $L-$Lipschitz continuous with respect to $\|\cdot\|_{\infty}$. Then $f$ has a representation (\ref{eq:Support-quasiconcave}) with $\|a_{j}\|_{1}\leq L$ for all $j\in\mathcal{J}$.

(iii) If $f$ has a representation (\ref{eq:Support-quasiconcave}), then
it is quasiconcave.
Moreover, if $a_j\geq 0$ for all $j\in\mathcal{J}$, then $f$ is non-decreasing. Conversely,
if $f$ is non-decreasing and quasiconcave, then there exists a set of kinked majorants $\{h_j\}_{j \in {\cal J}}$ with $a_j\geq 0$ such that representation (\ref{eq:Support-quasiconcave}) holds.

(iv) For any finite set $\Theta\subset \mathbb{R}^{d}$ and values $\left\{ v(\theta)\right\} _{\theta\in\Theta}\subset\mathbb{R}$,
$\hat{f}\text{ : }\mathbb{R}^{d}\rightarrow\mathbb{R}$ defined by
\begin{align*}
\hat{f}\left(x\right)\triangleq\inf_{a,\,b} \quad & b 
\nonumber
\\
\text{s.t.} \quad & \max\left\{ \langle a,\,\theta - x\rangle,\,0\right\} + b \geq v(\theta),\,\forall\theta\in\Theta,\\
& \|a\|_{1} \leq L,\nonumber
\end{align*}
is quasiconcave. 
Furthermore, the graph of $\hat{f}$ is the (pointwise) minimum of all $L-$Lipschitz quasiconcave majorants
of $\left\{ \left(\theta,\,v(\theta)\right)\text{ : }\theta\in\Theta\right\} $. 
\end{theorem}

We interpret $f \in \mathfrak{F}$ as an evaluation function for $G(z)$, $f(G(z))$ can be understood as the score for $G(z)$.
In particular, $f$ reflects our preferences over outcomes in ${\cal X}$.
Then, we optimize over $G(z)$ over $z \in {\cal Z}$ by finding the one which is maximal with respect to $f$.

\begin{proposition}
\label{prop:quasiconcave}
Let $f \in \mathfrak{F}$. Then $\rho_f : {\cal Z} \rightarrow \mathbb{R}$ defined by $\rho_f(z) = f(G(z))$ for all $z \in {\cal Z}$ is quasiconcave.
\end{proposition}

Based on Proposition~\ref{prop:quasiconcave}, we focus on optimizing $z \rightarrow f(G(z))$ with respect to monotone quasiconcave functions $f \in \mathfrak{F}$.
If $f \in \mathfrak{F}$ were fully specified, we would solve:
\begin{equation*}
\mathscr{P}(f) : \max_{z \in {\cal Z}} f(G(z)).
\end{equation*}
Even under perfect information, this class of problems is non-convex. 
However, by Proposition~\ref{prop:quasiconcave}, the objective of $\mathscr{P}(f)$ is a quasiconcave maximization problem for any $f \in \mathfrak{F}$, which retains enough structure for efficient optimization (e.g., by binary search).

\section{Robust Optimization Problem}
\label{sec:robust}

We focus on the case where we want to optimize $G(z)$ but we only have partial information about the evaluation function $f$.
For instance, when $G$ is multivariate there is often ambiguity about the marginal contribution to the score of each component as a function of the others.
The choice of evaluation function $f$ amounts to scalarizing the multi-objective optimization problem over $G(z)$.
We let ${\cal F} \subset \mathfrak{F}$ denote an ambiguity set of possible quasiconcave functions. This set is a user input, and it reflects the available partial information.

Given an ambiguity set ${\cal F} \subset \mathfrak{F}$, we define the pointwise worst-case valuation function $\psi_{{\cal F}} : {\cal X} \rightarrow \mathbb{R}$ via:
$\psi_{{\cal F}}(x) \triangleq \inf_{f\in {\cal F}}f(x),\, \forall x \in {\cal X}$.
Here, $\psi_{{\cal F}}$ is the worst-case valuation of an outcome $x$. Given $x \in {\cal X}$, we call $f^* \in {\cal F}$ such that $f^*(x) = \psi_{{\cal F}}(x)$ a worst-case valuation function 
at $x$. When the ambiguity set ${\cal F}$ is compact, the infimum is attainable.

Let $\wedge$ denote the point-wise minimum operation between $f, g \in \mathfrak{F}$ where $(f \wedge g)(x) \triangleq \min\{f(x), g(x)\}$ for all $x \in {\cal X}$.
We say that ${\cal F}$ is closed under $\wedge$ when $f, g \in {\cal F}$ imply $f \wedge g \in {\cal F}$.
Quasiconcavity, monotonicity, and Lipschitz continuity are all preserved by the operation $\wedge$ as summarized by the next proposition.

\begin{proposition}
Let ${\cal F} \subset \mathcal{F}_{\text{QCo}}$. Then: (i) $\psi_{{\cal F}} \in \mathcal{F}_{\text{QCo}}$; and (ii) if ${\cal F}$ is closed in the topology induced by the infinity norm under $\wedge$, then $\psi_{{\cal F}} \in {\cal F}$.
\end{proposition}

Our RO problem is:
\begin{equation}
\label{prob:robust}
\mathscr{P}({\cal F}) : \max_{z \in {\cal Z}} \psi_{{\cal F}}(G(z)) \equiv \max_{z \in {\cal Z}} \inf_{f \in {\cal F}} f(G(z)).
\end{equation}
For general ${\cal F}$, we have to write `inf' on the RHS of Eq.~\eqref{prob:robust} since the minimum may not be attained. When $\mathcal{F}$ is closed and 
$\mathfrak{F}$ is relatively compact, then $\mathcal{F}$ is compact (this will be the case for the specific ambiguity set we construct).
Problem~\eqref{prob:robust} is a quasiconvex optimization problem since $\psi_{{\cal F}}$ is monotone and quasiconcave.
In principle, if we are able to calculate a value and an upper subgradient of $\min_{f \in {\cal F}} f(G(z))$ at each $z$, then we can use the level function method to solve the problem.
Alternatively, we can theoretically solve Problem~\eqref{prob:robust} by doing binary search on its upper level sets and solving a sequence of convex optimization problems.
{\color{black}Specifically, since $z \rightarrow \psi_{{\cal F}}(G(z))$ is continuous and ${\cal Z}$ is compact,
there exist constants $v_{\min}$ and $v_{\max}$ such that
$-\infty < v_{\min} \leq \psi_{{\cal F}}(G(z)) \leq v_{\max} < \infty$ for all $z \in {\cal Z}$.
In general we do not know the  constants a priori, in practice, we would start with conservative choices for both, take $v_{\min} = \psi_{{\cal F}}(G(z))$ for any initial feasible $z$ (since we are looking for the maximizer), and take $v_{\max}$ to be a conservative upper bound (so $v_{\max}$ would not be a tight bound in most cases).
For a candidate value $v \in [v_{\min}, v_{\max}]$, we can easily solve the feasibility problem
\begin{equation}
\label{eq:feasibility problem}
{\cal G}(v) : \{z \in {\cal Z} : \psi_{{\cal F}}(G(z)) \geq v\}.
\end{equation}
Since $\psi_{{\cal F}}(G(z))$ is quasiconcave (as the composition of an increasing quasiconcave function $\psi_{{\cal F}}(\cdot)$ with a concave one $G(\cdot)$), the set of solutions to ${\cal G}(v)$ is a convex set. Then, when the set of solutions is non-empty, we can efficiently find a feasible solution. If ${\cal G}(v)$ has a solution $z^*$, then we take $v \Leftarrow (v_{\max} + \psi_{{\cal F}}(G(z^*)))/2$. Otherwise, we take $v \Leftarrow (v_{\min} + v)/2$.
In the forthcoming discussions (Theorem 5.3), we will show that 
binary search will 
terminate in finite (at most in $O(\log J)$ where $J$ is the number of data points) iterations.
The main computational difficulty 
comes down to solving the feasibility problem \eqref{eq:feasibility problem}
where we do not have a closed form for the 
function $\psi_{{\cal F}}(G(z))$.
This prompts us to develop an explicit representation of the upper level sets of $\psi_{{\cal F}}(G(z))$ through kinked majorants (Theorem~\ref{thm:upper_level}).
}
{\color{black} To this end, we consider a 
specifically structured ambiguity set ${\cal F}$.
}

{\color{black}
\begin{definition}
\label{Def-D-set}
   Let $\Theta = \{ \theta_1, \ldots, \theta_J \} \subset {\cal X}$ be a set of test outcomes with known lower bounds $\hat{v}(\theta)$ on the value of the target function.
%
Let ${\cal F}\subset \mathscr{F}$ be such that for each $f\in {\cal F}$,
\begin{equation}
\label{eq:function_values}
f(\theta) \geq \hat{v}(\theta),\, \theta \in \Theta.
\end{equation}
%
Define the data sample $\mathsf{D} = \{(\theta, \hat{v}(\theta))\}_{\theta \in \Theta}$.
\end{definition}
}
Note that the values $\hat{v}$ in $\mathsf{D}$ do not necessarily come from a function in $\mathcal{F}_{\text{QCo}}$, so we write this requirement as a set of inequalities rather than equalities to ensure it always gives a non-empty ambiguity set.



\begin{example}\label{ex:pro}
In \cite{haskell2022preference}, $\mathsf{D}$ is constructed by solving a mixed-integer linear program (MILP) which assigns the values of the worst-case evaluation function on $\Theta$.
For each $\theta \in \Theta$, let $v(\theta) \in \mathbb{R}$ correspond to the value and $s(\theta) \in \mathbb{R}^N$ correspond to an upper subgradient of the target function at $\theta$.
{\color{black}
To clarify, here $v(\cdot)$ is not a function, we 
use $v(\theta)$ to denote the entry of the vetor $v$ corresponding to $\theta$. 
}
Then, let $v = (v(\theta))_{\theta \in \Theta} \in \mathbb{R}^J$ 
be a collection of values and $s = (s(\theta))_{\theta \in \Theta} \in \mathbb{R}^{J N}$. 
Let $\widehat{\Theta} = \Theta \times \Theta$ be the set of edges in $\Theta$. In addition, let $\mathscr{R} \subset \Theta \times \Theta$ be a set of pairs of inputs where we require $v(\theta) \geq v(\theta')$ for all $(\theta, \theta') \in \mathscr{R}$ (i.e., the value $v(\theta)$ at $\theta$ must be at least as large as the value $v(\theta')$ at $\theta'$ for all pairs $(\theta, \theta') \in \mathscr{R}$).
We then consider the MILP:
\begin{subequations}\label{prob:value_MILP}
\begin{eqnarray}
\min_{v,\,s}\, && \sum_{\theta \in \Theta}v(\theta)\label{prob:value_MILP-1}\\
\text{s.t.}\, && v(\theta)+\max\{ \langle s(\theta),\,\theta'-\theta\rangle,\,0\} \geq v(\theta'), \quad \forall\left(\theta,\,\theta'\right)\in\widehat{\Theta},\label{prob:value_MILP-2}\\
 && v(\theta) \geq v(\theta'), \quad  \forall\left(\theta,\,\theta'\right)\in \mathscr{R},\label{prob:value_MILP-3}\\
 && v(\theta) \geq \hat{v}(\theta), \quad \forall \theta \in \Theta,\label{prob:value_MILP-4}\\
 && \textcolor{black}{s(\theta) \geq 0,\, \|s(\theta)\|_{1} \leq L, \quad \forall \theta \in \Theta,} \label{prob:value_MILP-5}
\end{eqnarray}
\end{subequations}
{\color{black}Problem~\eqref{prob:value_MILP} is illustrated in Figure~\ref{fig:pro} where the blue dots represent the test outcomes and the red parts represent the kinked majorants of each $\theta$. The objective is to find the worst-case evaluation on $\Theta$. }Problem~\eqref{prob:value_MILP} will produce values that match a quasiconcave function due to {\color{black}Eq.~\eqref{prob:value_MILP-2}}. However, Problem~\eqref{prob:value_MILP} is non-convex and the MILP reformulation may require solving $O(2^J)$ LPs in the worst-case. \cite{jian2025efficient} develop an efficient sorting algorithm to solve Problem~\eqref{prob:value_MILP} which requires solving $O(J^2)$ LPs. 
\end{example}

\begin{figure}
    \centering
    \includegraphics[width=0.35\linewidth]{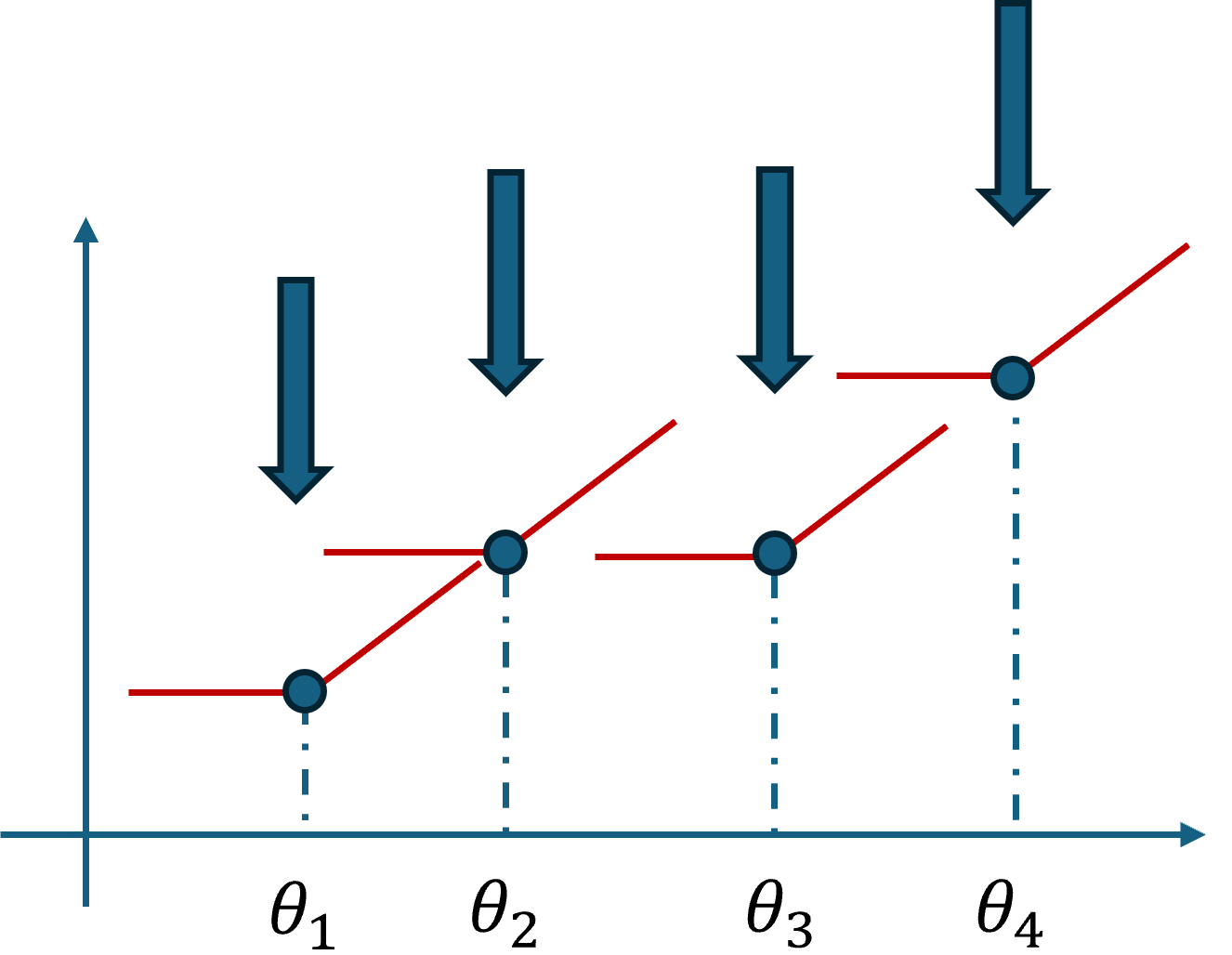}
    \caption{An illustration of Example~\ref{ex:pro} where $\Theta=\{\theta_1,\theta_2,\theta_3,\theta_4\}$ and $(\theta_2,\theta_3)\in\mathscr{R}$}
    \label{fig:pro}
\end{figure}

One notable feature of our framework is that $\mathsf{D}$ can essentially be arbitrary, and our procedure will still successfully construct the appropriate quasiconcave envelope.
If in fact $\mathsf{D}$ consists of points on the graph of a quasiconcave function, then $\psi_{{\cal U}}(\theta) = \hat{v}(\theta)$ will hold with equality for all $\theta \in \Theta$ by Theorem~\ref{thm:char-qco-function}(iv).
We focus on the particular ambiguity set
${\cal U} = {\cal U}(\mathsf{D}, L) = \Big\{f \in \mathcal{F}_{\text{QCo}} \cap {\cal F}_{\text{Lip}}(L) : f(\theta) \geq \hat{v}(\theta),\, \forall \theta \in \Theta\Big\}$,
based on the value assignments in Eq.~\eqref{eq:function_values} and a Lipschitz continuity requirement.
This set corresponds to the $L-$Lipschitz, monotone, quasiconcave envelope of the data $\mathsf{D}$.
{\color{black}Consequently, Problem~\eqref{prob:robust} can be written as}
\begin{equation}
\label{prob:robust_main}
\mathscr{P}({\cal U}) : \max_{z \in {\cal Z}} \psi_{{\cal U}} (G(z)).
\end{equation}
{\color{black}In the forthcoming discussions, we 
will develop efficient computational procedures to solve this optimization problem.
}

\section{Upper Level Sets}
\label{sec:upper}

In this section we develop the representation of $\psi_{{\cal U}}$ in terms of its upper level sets.

\begin{definition}
Let $f \in \mathfrak{F}$ and $\upsilon \in \mathbb{R}$. Then ${\cal A}(f,\upsilon) \triangleq \{x \in {\cal X} : f(x) \geq \upsilon\}$ is the upper level set of $f$ at level $\upsilon \in \mathbb{R}$.
\end{definition}

We say that a set ${\cal A} \subset \mathbb{R}^N$ is monotone if $x \in {\cal A}$ and $y \geq x$ \textcolor{black}{(for component-wise inequality)} imply \textcolor{black}{$y \in {\cal A}$}. By definition of $\mathcal{F}_{\text{QCo}}$, the upper level sets ${\cal A}(f,\upsilon)$ for all $f \in \mathcal{F}_{\text{QCo}}$ 
{\color{black} and $v\in\mathbb{R}$}
are monotone and convex. In fact, this is the essential feature that allows us to solve Problem~\eqref{prob:robust_main} by searching over the upper level sets.
The next proposition shows that any $f \in \mathfrak{F}$ is completely determined by its upper level sets. It also establishes a key relation for the upper level sets of the worst-case valuation $\psi_{\cal F}$.

\begin{proposition}
\label{prop:level}
(i) 
Let $f \in \mathfrak{F}$. Then
$f(x) = \sup\{\upsilon \in \mathbb{R} : x \in {\cal A}(f,\upsilon)\},\, \forall x \in {\cal X}$.

(ii) For $\mathcal{F} \subset \mathfrak{F}$, ${\cal A}(\psi_{\cal F}, \upsilon) = \cap_{f \in {\cal F}} {\cal A}(f, \upsilon)$ for all $\upsilon \in \mathbb{R}$.
\end{proposition}
\noindent
By Proposition~\ref{prop:level}(i), we can write $\mathscr{P}({\cal U})$ as:
$\max_{z \in \mathcal Z, \upsilon \in \mathbb{R}}\{ \upsilon : G(z) \in \mathcal{A}(\psi_{{\cal U}}, \upsilon) \}.$
This formulation is beneficial because we can then solve $\mathscr{P}({\cal U})$ {\color{black} by solving}
a sequence of convex feasibility problems of the form:
\begin{equation}
\label{prob:feasibility}
    \{ z \in {\cal Z} : G(z) \in \mathcal{A}(\psi_{{\cal U}}, \upsilon) \},
\end{equation}
where $\upsilon \in \mathbb{R}$ is fixed. Finding the largest $\upsilon$ for which Eq.~\eqref{prob:feasibility} is feasible is equivalent to solving $\mathscr{P}({\cal U})$.
Once we obtain the explicit form of $\{\mathcal{A}(\psi_{{\cal U}}, \upsilon)\}_{\upsilon \in \mathbb{R}}$, we can solve Eq.~\eqref{prob:feasibility} (or determine it is infeasible) for any $\upsilon$.

We now construct the upper level sets $\{{\cal A}(\psi_{{\cal U}}, \upsilon)\}_{\upsilon \in \mathbb{R}}$ for the specific ambiguity set ${\cal U}$ (for all levels). For this construction, we arrange our dataset $\mathsf{D}$ into a particular form without loss of generality.



\begin{definition}
For all $j \in [J]$, let $\mathcal{D}_j \triangleq \{ (\theta_1,\hat{v}(\theta_1)), \ldots,(\theta_j,\hat{v}(\theta_j))\}$ be the top $j$ inputs, in decreasing order of $\hat{v}-$value where $\hat{v}(\theta_1) \geq \hat{v}(\theta_2) \geq \cdots \geq \hat{v}(\theta_j)$. We say $\hat{v}(\theta_j)$ is the minimum $\hat{v}-$value in $\mathcal{D}_j$.
\end{definition}
\noindent
Under this convention, $\theta_1$ is the most preferred input based on the lower bounds with value $\upsilon_{\max} \triangleq \max\{\hat{v}(\theta) : \theta \in \Theta\} = \hat{v}(\theta_1)$.
Let $\mathcal{D}_{J} = \{(\theta_j, \hat{v}(\theta_j))\}_{j \in [J]}$ be the complete data set of inputs and value assignments, where $\hat{v}(\theta_1) \geq \hat{v}(\theta_2) \geq \cdots \geq \hat{v}(\theta_J)$.
There are no conditions on the value assignment $\hat{v}$ in $\mathsf{D}$, as long as the observations are sorted correctly in $\mathcal{D}_J$ in decreasing order of the lower bounds. Our approach will construct the quasiconcave envelope given any $\mathsf{D}$, once it is arranged into a sorted $\mathcal{D}_J$.

We next give an optimization formulation that computes the kinked majorant of $\mathcal{D}_j$.
This is part of the construction of the upper level sets of $\psi_{{\cal U}}$.
This problem is:
\begin{subequations}
\label{prob:interpolation}
\begin{align}
\mathscr{P}(x; \mathcal{D}_j) : \min_{\upsilon \in \mathbb{R}, \xi \in \mathbb{R}^N} \quad & \upsilon\\
{\rm s.t.} \quad & \upsilon + \max\{\langle \xi, \theta - x \rangle, 0\} \geq \hat{v}(\theta),\, \forall \theta \in \mathcal{D}_j,\label{prob:interpolation-2}\\
& \xi \geq 0,\, \|\xi\|_1 \leq L.
\end{align}
\end{subequations}
\underline{We slightly abuse notation and write $\theta \in \mathcal{D}_j$ to mean $(\theta,\hat{v}(\theta)) \in \mathcal{D}_j$}.
Problem~\eqref{prob:interpolation} finds the minimal $L-$Lipschitz kinked majorant $h(y) = \upsilon + \max\{\langle \xi, y - x \rangle, 0\}$ that dominates $\mathcal{D}_j$.
Problem $\mathscr{P}(x; \mathcal{D}_j)$ appears as part of our efficient algorithm in \cite{jian2025efficient} for solving Problem~\eqref{prob:value_MILP}.

Our construction is based on a level selection rule which maps a level $\upsilon$ to a subset of inputs in $\Theta$, which is then used to formulate the corresponding ${\cal A}(\psi_{{\cal U}}, \upsilon)$.  Since $\psi_{{\cal U}}$ is the quasiconcave envelope of $\mathsf{D}$, it will not attain any levels larger than $\upsilon_{\max}$.
We define $\hat{v}(\theta_{J+1}) = -\infty$ for the fictional input $\theta_{J+1}$ to make sure our upper level sets are well-defined for all levels (i.e., no specific $\theta_{J+1}$ ever actually appears in our procedure).

\begin{definition}[Level selection]
\label{Def:level-select}
Given $\upsilon \in (-\infty, \upsilon_{\max}]$, define $\kappa(\upsilon) \triangleq \{j\in [J] \mid \hat{v}(\theta_{j+1}) < \upsilon \leq \hat{v}(\theta_j)\}$. 
\end{definition}
\noindent
%
Essentially, the selection rule identifies only those points in $\mathsf{D}$ that matter in determining $\psi_{{\cal U}}(x)$, as shown in the following proposition.

\begin{proposition}
\label{prop:disjunctive}
Fix $x \in \mathbb{R}^N$ and let $\upsilon = \psi_{{\cal U}}(x)$. Then $\psi_{{\cal U}}(x) = \textsf{val}(\mathscr{P}(x; \mathcal{D}_j))$ for $j = \kappa(\upsilon)$.
\end{proposition}
\noindent
As a consequence of Proposition~\ref{prop:disjunctive}, by solving $\mathscr{P}(x; \mathcal{D}_j)$ for index $j = \kappa(\upsilon)$, we obtain a 
 kinked majorant of $\psi_{{\cal U}}$ at $x$.
So, it is also possible to apply the level function method proposed in \cite{xu2001level}.

Now we define an optimization problem based on the affine majorant. For given $j \in [J]$ and a candidate point $x \in {\cal X}$, we define:
\begin{subequations}
\label{eq:descent}
\begin{align}
\mathscr{P}_{LP}(x;\,\mathcal{D}_j) : \min_{\upsilon \in \mathbb{R}, \xi \in \mathbb{R}^N} \quad & \upsilon \label{eq:descent-1}\\
\textrm{s.t.} \quad & \upsilon+\left\langle \xi, \theta-x\right\rangle \geq \hat{v}(\theta),\,\forall \theta \in \mathcal{D}_j, \label{eq:descent-2}\\
& \xi \geq 0,\left\|\xi\right\|_{1} \leq L. \label{eq:descent-3}
\end{align}
\end{subequations}
Problem $\mathscr{P}_{LP}(x;\,\mathcal{D}_j)$ finds the smallest $L-$Lipschitz affine majorant $h(y) = \upsilon + \langle \xi, y - x \rangle$ at $x$ that dominates the values in $\mathcal{D}_j$.
Equivalently, $\mathscr{P}_{LP}(x;\,\mathcal{D}_j)$ returns the smallest $L-$Lipschitz affine majorant at $x$ that dominates the convex hull of $\mathcal{D}_j$.

\begin{proposition}\label{prop:interpolation-value-check}
Fix $x \in \mathbb{R}^N$ and let $\upsilon = \psi_{{\cal U}}(x)$. Then $\psi_{{\cal U}}(x) = \min\{\hat{v}(\theta_j), \textsf{val}(\mathscr{P}_{LP}(x; \mathcal{D}_j))\}$ for $j = \kappa(\upsilon)$.
\end{proposition}

The previous proposition leads to the characterization of ${\cal A}(\psi_{{\cal U}}, \upsilon)$ by linear programming duality, since we can use $\mathscr{P}_{LP}(x;\,\mathcal{D}_j)$ to check membership of the upper level sets of $\psi_{{\cal U}}$.
In particular, given $x \in {\cal X}$, we can check if $x \in {\cal A}(\psi_{{\cal U}}, \upsilon)$ by solving $\mathscr{P}_{LP}(x;\,\mathcal{D}_j)$ for $j = \kappa(\upsilon)$ and computing $\psi_{{\cal U}}(x) = \min\{\hat{v}(\theta_j), \textsf{val}(\mathscr{P}_{LP}(x; \mathcal{D}_j))\}$.

The following result gives the upper level sets for $\psi_{{\cal U}}$.
Let ${\vec 1}_{N} \in \mathbb{R}^{N}$ be the vector with all components equal to one.
We define translations of the inputs in $\Theta$ by their values according to $\tilde{\theta}_j\triangleq\theta_j-(\hat{v}(\theta_j)/L){\vec 1}_{N}$ (where we subtract $\hat{v}(\theta_j)/L$ from $\theta_j$ component-wise) for all $j \in [J]$.
{\color{black} This translation is a convenient way to rewrite the characterization of the upper level sets obtained from the dual of $\mathscr{P}_{LP}(x; {\cal D}_j)$. In particular, the Lipschitz continuity constraint $\|\xi\|_1\le L$ is the driver of the additive $(\upsilon/L){\vec 1}_N$ term in Theorem~\ref{thm:upper_level}. This effect is equivalently represented by shifting each sample point $\theta_j$ to $\tilde\theta_j$.}
Unlike level function or subgradient methods, which need the function value and an upper gradient at a point $x$, we derive the representation for the entire upper level sets of $\psi_{{\cal U}}$ over all ${\cal X}$.

\begin{theorem}\label{thm:upper_level}
The upper level set of $\psi_{{\cal U}}$ at level $\upsilon \in (-\infty, \upsilon_{\max}]$ satisfies:
\begin{equation*}
\mathcal{A}(\psi_{{\cal U}}, \upsilon) = \left\{x \in {\cal X} \;\middle|\; {\color{black} \exists \, p \in \mathbb{R}^{\kappa(\upsilon)}_{\geq 0} \text{ such that }} x \geq  \sum_{\theta\in \mathcal{D}_{\kappa(\upsilon)}} \tilde{\theta} \cdot p_\theta + (\upsilon/L) {\vec 1}_{N} ,\,
\sum_{\theta\in \mathcal{D}_{\kappa(\upsilon)}} p_\theta =1\right\}.
\end{equation*}
\end{theorem}
\noindent
By Theorem~\ref{thm:upper_level}, the upper level sets of $\psi_{{\cal U}}$ all have a polyhedral structure. 
Moreover, due to our finite data sample $\mathsf{D}$ and {\color{black}Propositions~\ref{prop:level} - \ref{prop:interpolation-value-check}}, 
$\psi_{{\cal U}}$ is piecewise linear. {\color{black} Figure~\ref{fig:upper-level-set}b in Section~\ref{sec:numerical} illustrates such upper level sets of the worst-case function with two inputs.}

Checking the inclusion $x \in {\cal A}(\psi_{{\cal U}}, \upsilon)$ can be done by solving a linear feasibility problem.
In view of $\tilde{\theta}$ as a translation of $\theta$, we can interpret ${\cal A}(\psi_{{\cal U}}, \upsilon)$ as the smallest monotone polyhedron that contains the convex hull of all translations $\tilde{\theta}$ of $\theta \in \Theta$ for which $\hat{v}(\theta) \geq \upsilon$.
If a given $\theta_j$ has a very low assigned value $\hat{v}(\theta_j)$, then its translation $\tilde{\theta}_j$ is not very ``deep''. It will be contained in the convex hull of other $\tilde{\theta}_{j'}$ and not change the form of the upper level set.

\section{Binary Search}
\label{sec:binary}

{\color{black} We now turn attention back to Problem~\eqref{prob:robust_main}, where we are optimizing an ambiguous quasiconcave objective over a convex feasible region. The previous Section~\ref{sec:upper} focused on characterizing the upper level sets of $\psi_{\cal U}$. Under our concavity assumption on $G(z)$, we can check its membership in the upper level sets of $\psi_{\cal U}$ by solving a convex optimization problem. This feature is what makes binary search possible for our problem.}

We now discuss an algorithm to solve Problem~\eqref{prob:robust_main}. Traditional cutting plane, level function, or subgradient descent methods can in principle be applied to solve this problem. \textcolor{black}{These include, for example, supporting-hyperplane/cutting-plane-type methods and subgradient-type methods developed for quasiconvex/quasiconcave optimization.} However, there are two difficulties with applying these traditional methods to Problem~\eqref{prob:robust_main}. First, all of them require us to evaluate $\psi_{{\cal U}}$ many times (which requires solving a hard non-convex optimization problem each time) to obtain the necessary function value and first-order information. Second, all of these traditional methods have asymptotic convergence guarantees, but not convergence rate results for Problem~\eqref{prob:robust_main}. \textcolor{black}{In our setting, this repeated evaluation for $\psi_{\cal U}$ becomes the main computational bottleneck.}
Instead, we use a binary search algorithm which searches over the upper level sets of $\psi_{{\cal U}}$. Theorem~\ref{thm:upper_level} shows that the upper level sets of $\psi_{{\cal U}}$ have a polyhedral structure that depends on the level $\upsilon$ through $\kappa(\upsilon)$, and that they can be explicitly characterized. This means that to solve Problem~\eqref{prob:robust_main} we only have to solve a sequence of convex feasibility problems, and we obtain the explicit logarithmic convergence rate inherent to binary search.

Because of the mapping $\kappa(\upsilon)$, the search over all $\upsilon \in (-\infty, \upsilon_{\max}]$ can be broken up into separate problems for each $j \in [J]$. Recall $\tilde{\theta}_j\triangleq\theta_j-(\hat{v}(\theta_j)/L){\vec 1}_{N}$ for each $j \in [J]$, then we define:
{\color{black}
\begin{subequations}
\label{prob:robust-G-D-j}
\begin{align}
\mathscr G(\mathcal{D}_j) : \max_{z \in {\cal Z},\, \upsilon \in \mathbb{R},\, p \in \mathbb{R}_{\geq 0}^j} \quad & \upsilon\\
\text{s.t.} \quad & G(z) \geq \sum_{\theta\in \mathcal{D}_j} \tilde{\theta} \cdot p_\theta + (\upsilon/L) {\vec 1}_{N},\\
& \sum_{\theta\in \mathcal{D}_j} p_\theta =1.
\end{align}
\end{subequations}}
The form of the constraints of $\mathscr G(\mathcal{D}_j)$ follows from the characterization of $\mathcal{A}(\psi_{{\cal U}}, \upsilon)$ in Theorem~\ref{thm:upper_level}. We make an immediate observation that $\textsf{val}(\mathscr G(\mathcal{D}_{j}))$ is monotone in $j$ since any feasible solution for $\mathscr G(\mathcal{D}_{j'})$ for $j' \leq j$ can be extended to a feasible solution for $\mathscr G(\mathcal{D}_{j})$ with the same optimal value.
\begin{observation} \label{obs:monotone}
For any $1\leq j'\leq j \leq J$, $\textsf{val}(\mathscr G(\mathcal{D}_{j'})) \leq \textsf{val}(\mathscr G(\mathcal{D}_{j}))$.
\end{observation}
We want to find the largest value of $\upsilon$ such that $G(z) \in \mathcal{A}(\psi_{{\cal U}}, \upsilon)$ for some $z \in {\cal Z}$.
However, $\mathscr G(\mathcal{D}_j)$ depends on $j$ so we need to identify the correct $j$ to determine the upper level set.
When the optimal value satisfies $\hat{v}(\theta_{j+1}) < \upsilon \leq \hat{v}(\theta_{j})$, then we have identified the correct index $j$ for $\mathscr G(\mathcal{D}_j)$.

\begin{proposition}\label{prop:binary}
Choose level $\upsilon \in (-\infty, \upsilon_{\max}]$ and $j = \kappa(\upsilon)$. Then $\max_{z\in{\cal Z}} \psi_{{\cal U}}(G(z))\geq \upsilon$ if and only if $\textsf{val} (\mathscr G(\mathcal{D}_j))\geq \upsilon$.
\end{proposition}

Observation~\ref{obs:monotone} and Proposition~\ref{prop:binary} justify using a binary search method to solve $\mathscr{P}({\cal U})$.
The following algorithm searches over the level $v$ between $\textsf{val}(\mathscr G(\mathcal{D}_{J}))$ and $\textsf{val}(\mathscr G(\mathcal{D}_{1}))$.
We give the details of 
the binary search procedure in Algorithm~\ref{algo:binary}.
Recall $\lfloor \cdot \rfloor$ is the floor function which returns the largest integer smaller than or equal to its argument.

\begin{algorithm}
\SetAlgoLined
 Initialization: sorted data sample $\mathcal{D}_{J} = \{(\theta_j, \hat{v}(\theta_j))\}_{j \in [J]}$, $j_1 = J$, $j_2 = 1$;
 
 \While{$j_1\ne j_2$}{
  Set $j: = \lfloor \frac{j_1+j_2}{2} \rfloor$, and compute $\upsilon_{j} = \textsf{val}(\mathscr G(\mathcal{D}_{j}))$ with optimal solution $z^*$
  ;
  
  \lIf{$\upsilon_{j} \leq \hat{v}(\theta_{j+1})$}{
   set $j_2:=j+1$}
   \lElse{
   set $j_1:=j$}
 }
 Set $j: = \lfloor \frac{j_1+j_2}{2} \rfloor$, and compute $\upsilon_{j} = \textsf{val}(\mathscr G(\mathcal{D}_{j}))$ with optimal solution $z^*$;
 
 \Return $z^*$ and $\psi_{{\cal U}}(G(z^*)) = \min\{\upsilon_{j}, \hat{v}(\theta_j)\}$.

 \caption{Binary search for $\mathscr{P}({\cal U})$}\label{algo:binary}
\end{algorithm}

Algorithm~\ref{algo:binary} searches over the levels $\{\hat{v}(\theta_1),\ldots,\hat{v}(\theta_J)\}$ of $\psi_{{\cal U}}$, where its upper level sets change form at the transition between each pair of distinct values in this set. Alternatively, we can view Algorithm~\ref{algo:binary} as a search over the space of indexes, where at each level we solve an instance of the convex optimization problem $\mathscr{G}(\mathcal{D}_j)$.
The complexity of solving $\mathscr{G}(\mathcal{D}_j)$ depends on the specifics of $G$. In general, $\mathscr{G}(\mathcal{D}_j)$ has $T+j+1$ decision variables, $N$ nonlinear inequality constraints, and one linear equality constraint. The number of constraints in $\mathscr{G}(\mathcal{D}_j)$ does not change as $j$ increases and uses more data, and the additional decision variables enter linearly, so $\mathscr{G}(\mathcal{D}_j)$ are all of comparable complexity for different values of $j$.

\begin{theorem}\label{thm:binary}
Algorithm~\ref{algo:binary} returns an optimal solution $z^*$ of $\mathscr{P}({\cal U})$, after solving at most $\log J$ instances of Problem~\eqref{prob:robust-G-D-j}.
\end{theorem}

The computational complexity of $\log\,J$ in Theorem \ref{thm:binary} follows from the standard complexity of binary search (since we have at most $J$ unique levels to check). In line with this theorem, the running time of Algorithm \ref{algo:binary} is not really sensitive to $J$. It is more sensitive to the problem dimension $N$ and the complexity/nonlinearity of the stochastic function $G$, since these features determine how hard it is to solve each instance of Problem~\eqref{prob:robust-G-D-j}.
While level function methods and subgradient descent methods can be applied to $\mathscr{P}({\cal U})$, they do not enjoy the same logarithmic complexity as Algorithm \ref{algo:binary} (their convergence guarantees are usually asymptotic). In addition, evaluating $\psi_{{\cal U}}(x)$ for even a single $x$ is numerically challenging. Algorithm \ref{algo:binary} overcomes this issue by directly using the upper level sets of $\psi_{{\cal U}}$. A level set/subgradient descent method would also have to overcome this issue to obtain function values and upper subgradients.
In addition, binary search solves the quasiconcave maximization problem
precisely, whereas the level function and subgradient descent methods generally solve it only approximately.
The fact that we can use binary search at all is due to the specific structure of the worst-case quasiconcave function which lets us explicitly construct its upper level sets.

\section{Numerical Experiments}\label{sec:numerical}

{\color{black}
We apply our method to two resource allocation problems and report numerical experiments in this section. The first is a production efficiency problem with a quasiconcave Cobb-Douglas function. The second is a fair resource allocation problem where the individuals in a heterogeneous population have quasiconcave utility functions.

\subsection{Cobb-Douglas Production}
}

We suppose there are $N \geq 1$ inputs, where $x_n \geq 0$ is the amount we use of input $n \in [N]$. We have $I \geq 1$ resource constraints which determine the feasible region:
$
{\cal X} \triangleq \left\{ x \in \mathbb{R}^N_{\geq 0} : \sum_{n=1}^N a_{in} x_n \leq b_i, i \in [I], x_n \in [x_{\min}, x_{\max}], n \in [N] \right\},
$
where all $a_{in}, b_i > 0$ and $0 < x_{\min} \leq x_{\max} < \infty$ (so ${\cal X}$ is automatically nonempty, convex, and compact).
We recall the Cobb-Douglas production efficiency problem introduced in \cite{bradley1974fractional} given by:
%
\begin{equation}\label{eq:profit-cost-ratio}
\max_{x \in {\cal X}} f(x) \triangleq \frac{\alpha_0 \prod_{n=1}^N  x_n^{\alpha_n}}{\sum_{n=1}^N c_n x_n + c_0},
\end{equation}
%
where $\sum_{n=1}^N \alpha_n = 1$, which maximizes the ratio of production value (i.e., $\alpha_0 \prod_{n=1}^N x_n^{\alpha_n}$) to production cost (i.e., $\sum_{n=1}^N c_n x_n + c_0$). 
We assume that
$x_{\min} >0$ 
so that the objective function is Lipschitz continuous on ${\cal X}$. This problem is maximizing a quasiconcave function over a convex set. {\color{black} Note that the function $f$ is quasiconcave but not necessarily monotone.
We adapt our method to handle non-monotone objectives with only slight modification as described in Appendix~\ref{sec:nonmonotone}.}

{\color{black}The Cobb--Douglas production function is a classical and widely used model for the relationship between inputs (e.g., labor, capital, materials) and outputs in various industries, from manufacturing to agriculture to services.
It has been extensively applied in production economics and efficiency analysis.
Its practical relevance stems from its ability to capture diminishing returns to scale and its tractability in both theoretical analysis and empirical estimation.}

Suppose that the decision maker does not have complete knowledge about the true objective function in Eq.~\eqref{eq:profit-cost-ratio} (e.g., the value/cost parameters, or even the specific form of the production function). Nevertheless, the decision maker has several observations of the value/cost ratio from historical data of the total production quantity and the total cost for companies in this sector. Through these observations, the decision maker can form lower bounds on the value of the true objective function, i.e., $f(x) \geq \hat v(\theta)$ for $\theta \in \Theta$. In addition, the decision maker knows the objective function is quasiconcave and Lipschitz continuous. The decision maker would then like to generate a robust solution to Problem~\eqref{eq:profit-cost-ratio} based on available partial information.

For the remainder of this section, unless otherwise specified, we take $N=2$, $(\alpha_0, \alpha_1, \alpha_2) = (1.0, 0.6, 0.4)$, $(c_0, c_1, c_2)=(1.0, 1.0, 2.0)$, and $x_{\min} = 0.5$, $x_{\max} = 10$. For this specific choice of parameters, the function has a Lipschitz constant of $0.20$. However, we may not know the exact value of the Lipschitz constant since the true objective function is unknown. {\color{black} Theoretically, increasing $L$ weakens the smoothness restriction and enlarges the ambiguity set ${\cal U}$, which can make the worst-case objective more conservative. In contrast, decreasing $L$ imposes stronger regularity and typically reduces conservatism. In the experiments, we choose $L=0.30$ to be the conservative value to construct the worst-case objective. }

\subsubsection{Upper Level Sets}
In this subsection, we compare the upper level sets of the true objective function $f$ and its robust counterpart $\psi_{\mathcal{U}}$. In Figure~\ref{fig:sub1}, we plot the surface of the true function $f$. Next, we randomly sample $J=200$ points uniformly from $[0.5, 10] \times [0.5,10]$ and obtain their corresponding true function values. These values are used as lower bounds to generate a worst-case function $\psi_{\mathcal{U}}$. In Figure~\ref{fig:sub2}, we compare the upper level sets of the true function and $\psi_{\mathcal{U}}$. First, we observe that the upper level sets are all convex, verifying the quasiconcavity of both functions. {\color{black} The jagged edges observed in the plot are discretization artifacts arising from the finite mesh.} Second, we note that the upper level sets of $f$ are always contained in those of $\psi_{\mathcal{U}}$, which we would expect by definition of $\psi_{\mathcal{U}}$.

\begin{figure}[htbp]
  \centering

  \begin{subfigure}[t]{0.48\textwidth}
    \centering
    \includegraphics[width=0.7\linewidth]{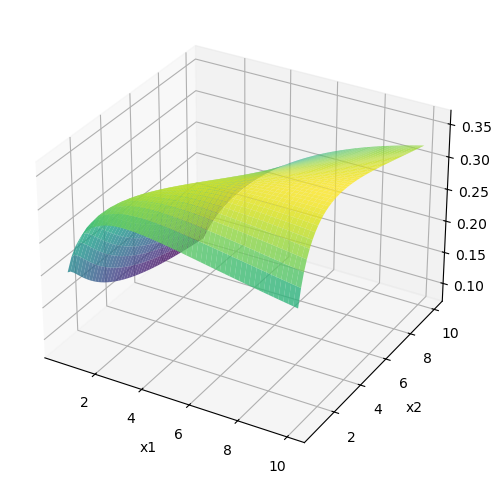}  
    \caption{True Function Surface}                           
    \label{fig:sub1}
  \end{subfigure}
  \hfill
  \begin{subfigure}[t]{0.48\textwidth}
    \centering
    \includegraphics[width=0.7\linewidth]{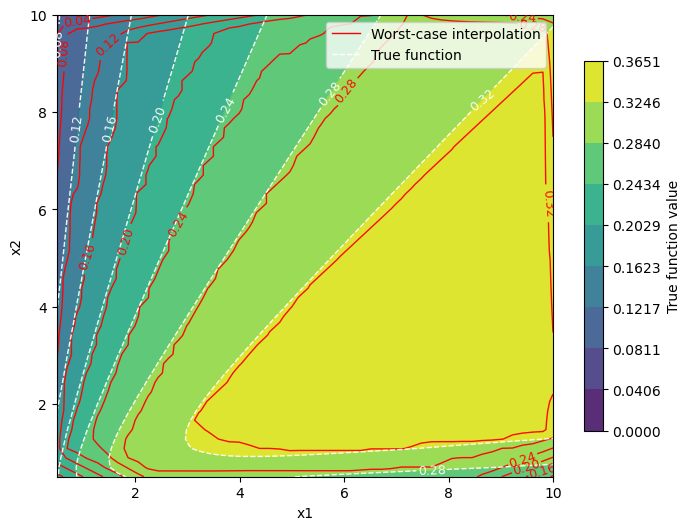}
    \caption{Worst-Case Function Contour Lines}
    \label{fig:sub2}
  \end{subfigure}

  \caption{Comparison of Upper Level Sets}
  \label{fig:upper-level-set}
\end{figure}


\subsubsection{Effectiveness of Robust Framework}
In this subsection, we demonstrate the effectiveness of our robust approach for an ambiguous objective function. First, we compare the quality of the output of our function approximation method against two benchmarks: (i) concave regression; and (ii) piecewise constant approximation where we interpolate between sampled points with constant functions. To compute (i), we first create linear function fits for clustered sampled points and then take the minimum of these linear functions. For (ii), we note that the upper level sets of the piecewise constant approximation are just the convex hulls of all sampled points with greater function values (and so the piecewise constant approximations are quasiconcave). However, concave regression necessarily gives a concave function which does not accurately reflect quasiconcavity of the target function in this case. 

In Figure~\ref{fig:comparion-diff-methods}, we plot the functions produced by the three methods using $J=200$ sample points generated uniformly from $[0.5, 10] \times [0.5,10]$, represented by the red dots. We observe that due to misspecification of concave regression for the Cobb-Douglas function, the resulting estimated function generally deviates from the sample points and does not well approximate the target function. We also observe that, compared to the piecewise constant approximation, 
our worst-case quasiconcave function is smooth due to the Lipschitz constraint. We see that the piecewise constant approximation is too conservative especially near the boundaries where the sample points are sparse. In summary, our worst-case quasiconcave function gives the most consistent approximation of the true Cobb-Douglas production function. 

\begin{figure}[htbp]
  \centering
  \begin{subfigure}[t]{0.3\textwidth}
    \centering
    \includegraphics[width=1\linewidth]{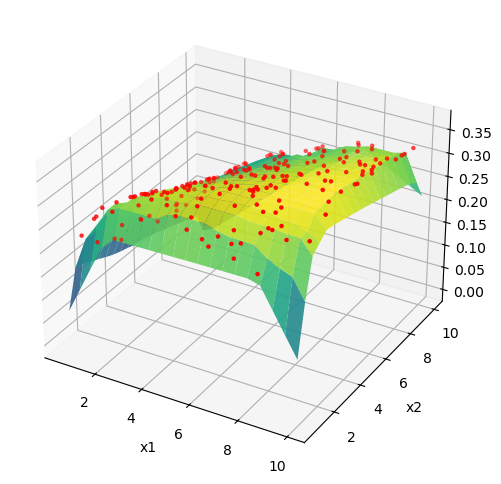}  
    \caption{Quasiconcave}                           
    \label{fig:compare-sub1}
  \end{subfigure}
  \hfill
  \begin{subfigure}[t]{0.3\textwidth}
    \centering
    \includegraphics[width=1\linewidth]{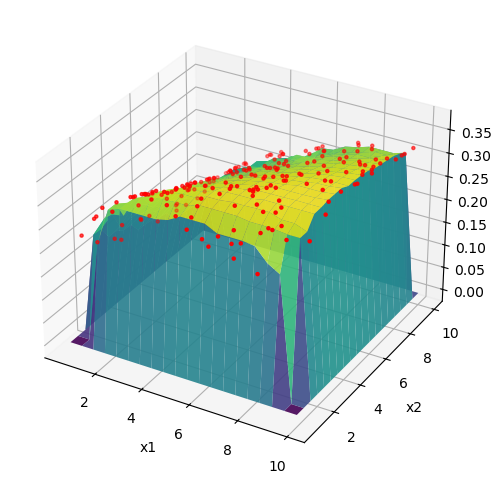}
    \caption{Piecewise Constant}
    \label{fig:compare-sub2}
  \end{subfigure}
  \hfill
  \begin{subfigure}[t]{0.3\textwidth}
    \centering
    \includegraphics[width=1\linewidth]{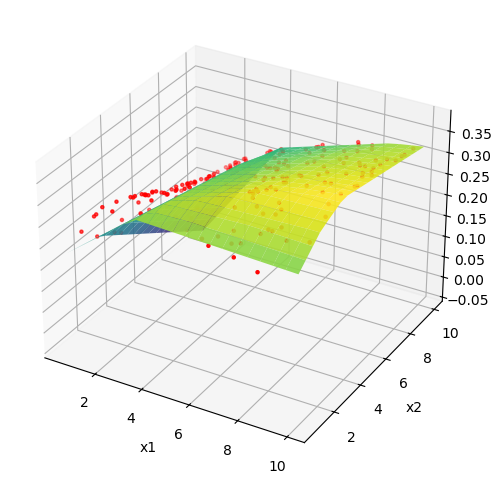}
    \caption{Concave}
    \label{fig:compare-sub3}
  \end{subfigure}

  \caption{Comparison of Function Approximations}
  \label{fig:comparion-diff-methods}
\end{figure}

To quantify the quality of the function approximation by different methods, we compare the $L_1$ norm between the approximated function and the true one for increasing sample sizes. We do not use the $L_\infty$ norm here because the $L_\infty$ norm of the piecewise constant approximation is constant even for large sample sizes, as seen in Figure~\ref{fig:compare-sub2}. In Figure~\ref{fig:L1}, we see that our worst-case quasiconcave function quickly converges to the target function as the sample size increases, while the piecewise constant function converges much more slowly. The concave regression, however, suffers from model misspecification and does not converge even with increasing sample size. {\color{black}Moreover, because concave regression assumes a concave target function while the true Cobb--Douglas function in our experiment is only quasiconcave, increasing the sample size can make the mismatch even more apparent. The best concave fit can shift as additional samples reveal more non-concave curvature, which can lead to non-monotone empirical $L_1-$errors for increasing $J$. This observation further highlights the importance of using the more flexible class of quasiconcave functions for regression when concavity is not justified.}  

\textcolor{black}{This behavior highlights the trade-off in choosing the sample size $J$. When $J$ is larger, the ambiguity set ${\cal U}$ is smaller because more majorization constraints $f(\theta)\geq \hat v(\theta)$ are imposed. This typically reduces the conservatism of the quasiconcave envelope, and leads to a less pessimistic worst-case objective. Conversely, many functions remain admissible when $J$ is small and the resulting envelope (and optimal solution) are more conservative. Subsequently, they are more robust to misspecification of the objective due to limited information.}

\begin{figure}
    \centering
    \includegraphics[width=0.55\linewidth]{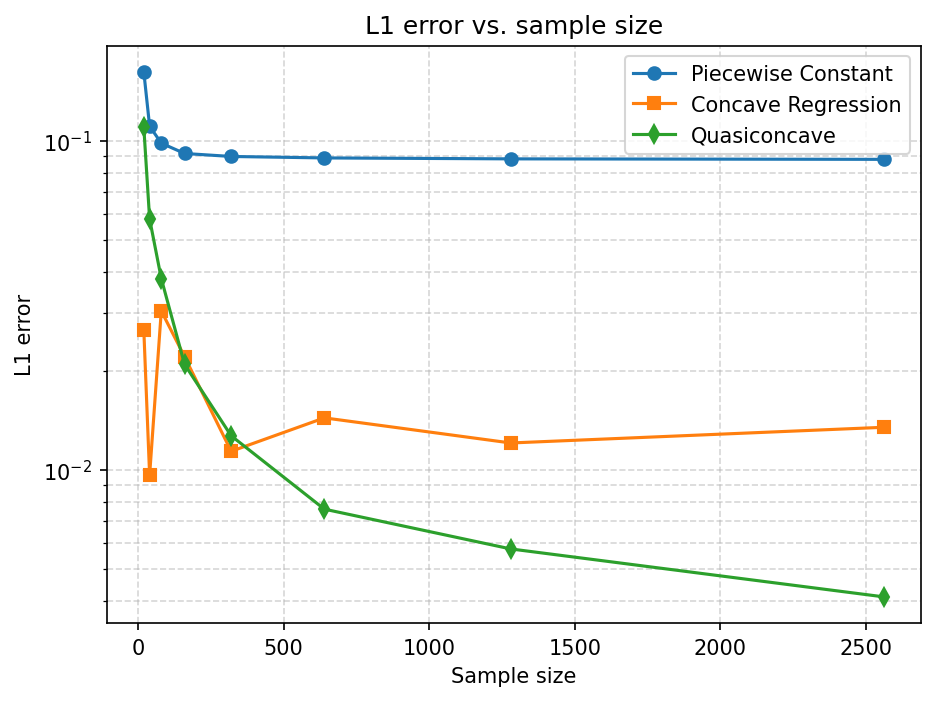}
    \caption{$L_1-$error of Different Function Approximation Methods}
    \label{fig:L1}
\end{figure}

Finally, we compare the performance of the solutions generated from these three methods. To obtain the theoretical optimal solution to Problem~\eqref{eq:profit-cost-ratio}, we reformulate it as a geometric programming problem which can be handled by standard solvers:
\begin{subequations}
\begin{align*}
\max_{x \in \mathbb{R}^N_{\geq 0}, t\geq 0} \quad & \alpha_0\prod_{n=1}^N  x_n^{\alpha_n}t^{-1}\\
{\rm s.t.} \quad & \sum_{n=1}^N a_{in} x_n \leq b_i, i \in [I],\\
& \sum_{n=1}^N c_n x_n + c_0 \leq t.
\end{align*}
\end{subequations}
In Table~\ref{table:optimality-gap}, we report the average percentage optimality gap of the solutions (i.e., the difference between the production-cost ratio of the solution and the theoretical optimal solution, divided by the true optimal value). The results in Table~\ref{table:optimality-gap} are based on 200 replications of the experiment. We also report the standard deviation and the maximum of the percentage optimality gap over these replications. 

Compared with the other methods, our robust quasiconcave solution consistently reaches the lowest optimality gap over different sample sizes. Moreover, it converges much faster than the solution from the piecewise constant approximation. This is because, although the piecewise constant approximation is quasiconcave, it does not account for the additional smoothness information. 

\begin{table}[ht]
\begin{centering}
{\footnotesize{
\begin{tabular}{lcccccc}
\toprule
& \multicolumn{6}{c}{Number of Sampled Points $J$}\\
 \cline{2-7}
Method  & 32 & 64 & 128 & 256 & 512 & 1024\\
\hline
Concave & 8.6 (14.6, 63.7) & 5.2 (6.6, 16.9) & 1.1 (1.5, 5.3) & 0.3 (0.4, 1.7) & 0.7 (1.3, 4.6) & 0.2 (0.2, 0.9) \\
QCO Constant  & 2.5 (3.4, 14.9) & 1.0 (1.0, 3.9) & 0.5 (0.5, 1.7) & 0.4 (0.3, 1.0) & 0.1 (0.2, 0.8) & 0.1 (0.1, 0.5) \\
QCO Lipschitz & 1.0 (1.4, 5.2) & 0.8 (1.0, 3.1) & 0.3 (0.5, 1.4) & 0.2 (0.2, 0.7) & 0.1 (0.2, 0.5) & 0.1 (0.1, 0.5) \\
\bottomrule
\end{tabular}
}}
\par\end{centering}
\caption{Optimality gap (in \%) of the solutions generated by different methods. {\color{black}``Concave'' is concave regression; ``QCO Constant'' is a piecewise constant approximation; and ``QCO Lipschitz'' is our robust Lipschitz-quasiconcave solution.} The numbers in parenthesis are the standard deviation and maximum.}
\label{table:optimality-gap}
\end{table}
\subsubsection{Efficiency of Binary Search Algorithm}
This subsection demonstrates the efficiency of our proposed method with respect to increasing problem scale. First, we consider the computation of the worst-case quasiconcave function value $\psi_{\mathcal{U}}(x)$. Notice that Algorithm~\ref{algo:binary} can be directly adapted to calculating $\psi_{\mathcal{U}}(x)$ simply by treating $x \equiv G(z)$ where $\mathcal{Z} = \{z\}$ is a singleton. For the benchmark, we solve for the worst-case value $\psi_{\mathcal{U}}(x)$ using the non-convex problem $\mathscr{P}(x; \mathcal{D}_J)$. We solve $\mathscr{P}(x; \mathcal{D}_J)$ as an MILP using the Big-M method for the disjunctive constraints, denoted as $\mathscr{P}_{MILP}(x; \mathcal{D}_J)$: 
\begin{subequations}
\begin{align*}
\min_{\upsilon \in \mathbb{R}, \xi \in \mathbb{R}^N, s_1, s_2 \in \mathbb{R}^J, u \in \{0,1\}^J } \quad & \upsilon\\
{\rm s.t.} \quad & \upsilon + s_1(\theta) \geq \hat{v}(\theta),\, \forall \theta \in \mathcal{D}_J,\\
& \langle \xi, \theta - x \rangle \geq s_1(\theta) - s_2(\theta), \, \forall \theta \in \mathcal{D}_J,\\
& s_1(\theta) \leq u(\theta) \cdot M, s_2(\theta) \leq (1-u(\theta)) \cdot M, \, \forall \theta \in \mathcal{D}_J,\\
& \xi \geq 0,\, \|\xi\|_1 \leq L,
\end{align*}
\end{subequations}
where {\color{black} $M=10^3$ is the Big-$M$ constant used in our implementation (chosen to safely dominate the range of $s_1(\theta)$ and $s_2(\theta)$ implied by the bounds in our experiments). All MILPs were solved with HiGHS (the default MILP solver of CVXPY). The disjunctive constraints could alternatively be formulated as SOS1 constraints.}

We report the time required to compute $\psi_{\mathcal{U}}(x)$ for a randomly selected $x$ with Algorithm~\ref{algo:binary} and with $\mathscr{P}_{MILP}(x; \mathcal{D}_J)$ in Table~\ref{table:scalability-samples}.
From Table~\ref{table:scalability-samples}, we observe that Algorithm~\ref{algo:binary} scales well with the sample size, while the MILP reformulation is much slower for large sample sizes. 

\begin{table}[ht]
\begin{centering}

\begin{tabular}{ccccccccccccc}
\hline
 & &  \multicolumn{7}{c}{Number of Sampled Points $J$}\tabularnewline
 \cline{3-10}

 & Method & 32 & 64 & 128 & 256 & 512 & 1024 & 2048 & 4096 \\
\hline 

& Binary Search & 0.041 & 0.082 & 0.174 & 0.371 & 0.851 & 1.894 & {\color{black} 4.035} & {\color{black} 9.235} \\
& MILP & 0.074 & 0.137 & 0.285 & 0.655 & 1.810 & 6.823 & {\color{black} 34.079} & {\color{black} 204.739} \\

\hline 
\end{tabular}
\par\end{centering}
\caption{Running time of Algorithm~\ref{algo:binary} compared to the MILP reformulation with increasing sample size and fixed dimension $N=2$ (in seconds).}
\label{table:scalability-samples}
\end{table}

Next, we test the scalability of Algorithm~\ref{algo:binary} for the robust optimization problem $\mathscr{P}(\mathcal{U})$. {\color{black} The MILP benchmark above is useful for evaluating $\psi_{\mathcal{U}}(x)$ at a fixed $x$, but it is not directly suitable for the overall robust optimization problem because the outer maximization would require repeatedly solving MILPs.} For a benchmark, we use the level function method for quasiconvex optimization proposed by \cite{xu2001level}. This method uses the subgradient of the objective function to construct a level function in each iteration. In addition, it is guaranteed to stop at the global optimum. The procedure is summarized in Algorithm~\ref{algo:level_function}. For further details about this method and its theoretical guarantees, the readers are referred to \cite{xu2001level}.

\begin{algorithm}
\SetAlgoLined

 Initialization: $x_0 \in \mathcal{X}$, $i=0$, $\sigma_{-1}(x) = \infty$, and $\Delta(0) = \infty$;
 
 \While{$\Delta(i)> 0$}{
  Let $(v, \xi, s_1, s_2, u)$ be the optimal solution of $\mathscr{P}_{MILP}(x_i; \mathcal{D}_J)$;
  
  Set level function $\sigma_{x_i}(x) := \langle \xi, x - x_i\rangle$ and set $\sigma_i(x) := \min \{\sigma_{i-1}(x), \sigma_{x_i}(x)\}$;
  
  Let $x_{i+1} \in \arg \max_{x \in \mathcal{X}} \sigma_i(x)$ and let $\Delta(i+1) := \sigma_i(x_{i+1})$;
  
  Set $i\leftarrow i+1$;

 }
 
 \Return $x_i$.

 \caption{Level Function Method for $\mathscr{P}(\mathcal{U})$}\label{algo:level_function}
\end{algorithm}

In Table~\ref{table:scalability-samples-RO}, we report the running time of Algorithm~\ref{algo:binary} versus Algorithm~\ref{algo:level_function} for increasing problem size. We see that Algorithm~\ref{algo:binary} is very efficient compared to Algorithm~\ref{algo:level_function}, whose running time increases rapidly with the sample size. This is because Algorithm~\ref{algo:level_function} has to solve an MILP in each iteration to obtain the level function, while our method only solves $\log(J)$ LPs. In particular, the other first-order methods (e.g., \cite{hu2015inexact,yu2019abstract,grad2023relaxed,hazan2015beyond}) all need to solve an MILP in each iteration to obtain sub-differential information. Thus, they all face the same bottleneck to their scalability.

\begin{table}[ht]
\begin{centering}

\begin{tabular}{ccccccccccc}
\hline
 & &  \multicolumn{7}{c}{Number of Sampled Points $J$}\tabularnewline
 \cline{3-9}

 & Method &  16 &  32 &  64 &  128 & 256 & 512 & 1024
\tabularnewline
\hline 

& Binary Search & 0.043 & 0.063 & 0.083 & 0.141 & 0.268 & 0.461 & 0.940\\
& Level Function & 1.711 & 4.328 & 6.659 & 10.592 & 20.222 & 93.370 & 925.354\\

\hline 
\end{tabular}
\par\end{centering}
\caption{Running time of the binary search algorithm compared to the level function method with increasing sampled points and fixed dimension $N=2$ (in seconds).}
\label{table:scalability-samples-RO}
\end{table}




\subsection{Max-Min Utility Problem}\label{subsec:maxmin-utility}
{\color{black}In this subsection, we consider a max-min utility problem that arises in resource allocation with fairness considerations. We are allocating resources $x \in {\cal X} \subset \mathbb{R}^N$ in a convex, compact feasible set among a heterogeneous population. Each individual $i \in {\cal I} = \{1, 2, \ldots, I\}$ has a utility function $u_i : {\cal X} \rightarrow \mathbb{R}$ which depends on the allocation. The goal is to maximize the utility of the \textit{worst-off} individual, which provides fairness guarantees on the resulting solution. This leads to the optimization problem:
\begin{equation}\label{eq:maxmin-utility}
\max_{x \in {\cal X}} \min_{i \in {\cal I}} u_i(x).
\end{equation}
\noindent
When all utility functions $\{u_i\}_{i \in {\cal I}}$ are quasiconcave, the minimum $\min_{i \in {\cal I}} u_i$ is also quasiconcave, and this problem fits into our robust quasiconcave optimization framework. \cite{bertsimas2011price} study the price of fairness for resource allocation problems based on proportional and max-min fairness (such as Problem~\eqref{eq:maxmin-utility}).}

{\color{black}In practice, the true utility functions $\{u_i\}_{i \in {\cal I}}$ typically cannot be fully specified a priori. Instead, we have access to limited partial information about them. For instance, we may conduct surveys or preference elicitation experiments to measure utility for set points (representing allocations). These data points, while sparse, can help us construct a useful ambiguity set for the actual utility functions. Our framework can leverage this partial information to make decisions that are robust to the uncertainty about the true utility functions.}

{\color{black}In our experiments, Problem~\eqref{eq:maxmin-utility} optimizes over the location $x \in \mathbb{R}^2$ of a facility (e.g., a hospital, school, or emergency response center).
Each individual $i$ is located at a fixed point $p_i \in \mathbb{R}^2$, and the utility of individual $i$ is a decreasing function of the Euclidean distance $d(x, p_i) = \|x - p_i\|_2$ from the facility location $x$ to their location $p_i$. We take the sigmoidal utility function:
$u_i(x) = 1 - (1 + \exp(-\kappa(d(x, p_i) - d_0)))^{-1}$,
where $\kappa > 0$ controls the steepness of the sigmoid and $d_0 > 0$ is a critical distance threshold. This utility function is quasiconcave in $x$ since it is a decreasing function of the convex distance function $d(x, p_i)$.}

{\color{black}Our numerical experiments use the following settings: the population size is $I = 10$; individual locations $p_i$ are drawn uniformly from the square $[2, 8] \times [2, 8]$; the sigmoid parameters are $\kappa = 1.5$ and $d_0 = 2.5$; the feasible region is ${\cal X} = \{x \in \mathbb{R}^2 : x_1, x_2 \in [0, 10],\, x_1+x_2\leq8\}$; and the Lipschitz constant is $L = 0.5$.
}

\subsubsection{Sensitivity to Sample Size}\label{subsubsec:maxmin-sensitivity}
{\color{black}Similar to the Cobb-Douglas experiment, we compare our robust quasiconcave approach against two benchmarks: (i) piecewise constant approximation, denoted `QCO Const'; and (ii) concave regression. We sample $J \in \{16, 32, 64, 128, 256, 512, 1024\}$ points uniformly from ${\cal X}$ and evaluate the $L_1-$error between the true objective and the corresponding approximations obtained from $100$ randomly sampled points from $\mathcal{X}$. Table~\ref{table:maxmin-L1} reports the average $L_1-$error for all three methods over 20 complete runs of the experiment. These numerical results show that our method (QCO Lipschitz) achieves lower $L_1-$error than piecewise constant approximation (QCO Const) and concave regression for every value of $J$. Furthermore, the $L_1-$error of our method consistently decreases as $J$ increases. In contrast, concave regression does not improve monotonically with $J$. Our approach enjoys better approximation accuracy and more stable improvement with respect to sample size $J$.}

\begin{table}[ht]
\begin{centering}
\begin{tabular}{lccccccc}
\toprule
& \multicolumn{7}{c}{Number of Sampled Points $J$}\\
 \cline{2-8}
Method & 16 & 32 & 64 & 128 & 256 & 512 & 1024\\
\hline
QCO Lipschitz & 0.0111 & 0.0115 & 0.0073 & 0.0079 & 0.0050 & 0.0039 & 0.0023 \\
QCO Const & 0.0138 & 0.0138 & 0.0090 & 0.0090 & 0.0061 & 0.0055 & 0.0032 \\
Concave Regression & 0.0147 & 0.0359 & 0.0411 & 0.0485 & 0.0248 & 0.0228 & 0.0249 \\
\bottomrule
\end{tabular}
\par\end{centering}
\caption{$L_1-$error between the approximate and true max-min objective functions for each method, versus number of sampled points $J$ (dimension $N=2$).
}
\label{table:maxmin-L1}
\end{table}

\subsubsection{Fairness Analysis}\label{subsubsec:maxmin-fairness}
{\color{black} We now want to see how robust optimization affects fairness guarantees when the true utility functions are ambiguous.
We analyze the distribution of utilities across individuals for the true (perfect information) optimal solution, the robust solution from our method, and the solutions from the benchmark methods (piecewise constant approximation and concave regression). 
Figure~\ref{fig:Maxmin_Solution_Map} shows the facility locations chosen by each method together with the individual locations $p_i$ for one run of this experiment. The solution from our method (QCO Lipschitz) is the closest to the true optimal location, whereas concave regression, building on incorrect structural information, is way off. We repeat this experiment for $20$ runs and for each solution we compute: (i) the minimum utility (worst-off individual); (ii) the mean utility across all individuals; (iii) the standard deviation of the population's utilities (which is a measure of inequality); and (iv) the price of fairness (POF) \citep{bertsimas2011price}, defined as the relative reduction in the sum of utilities under the fair allocation, compared to the utilitarian solution (to maximize the total utility). The average metrics over the $20$ repetitions are reported in Figure~\ref{fig:Maxmin_Fairness_Comparison}. Compared to the benchmarks, our method consistently achieves the fairest allocation in terms of minimum, mean, and standard deviation of the population's utilities, with little compromise of the system efficiency, measured by POF. 
}
\begin{figure}[h]
    \centering
    \includegraphics[width=0.40\linewidth]{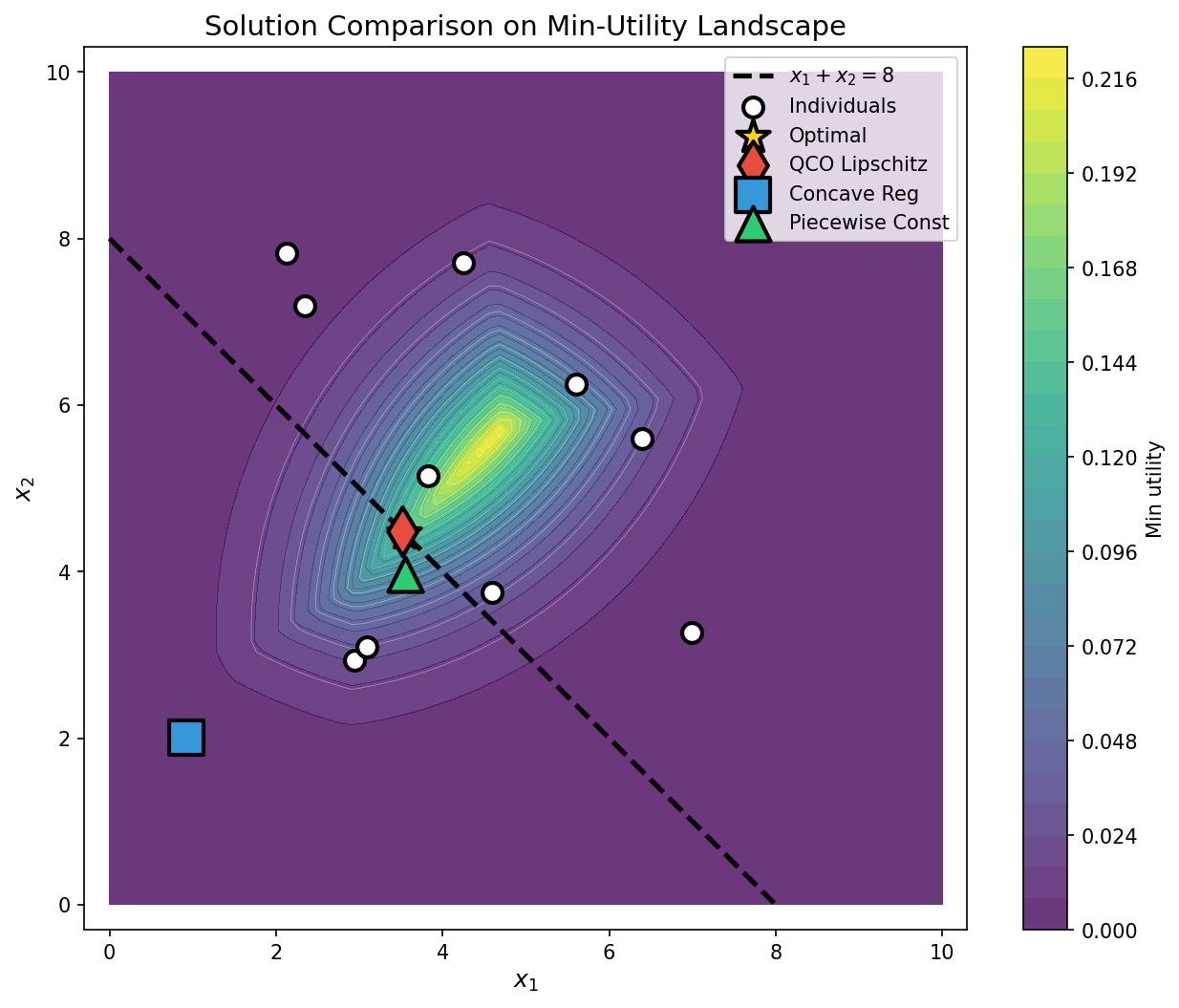}
    \caption{Locations of Different Solutions}
    \label{fig:Maxmin_Solution_Map}
\end{figure}

\begin{figure}[h]
    \centering
    \includegraphics[width=0.8\linewidth]{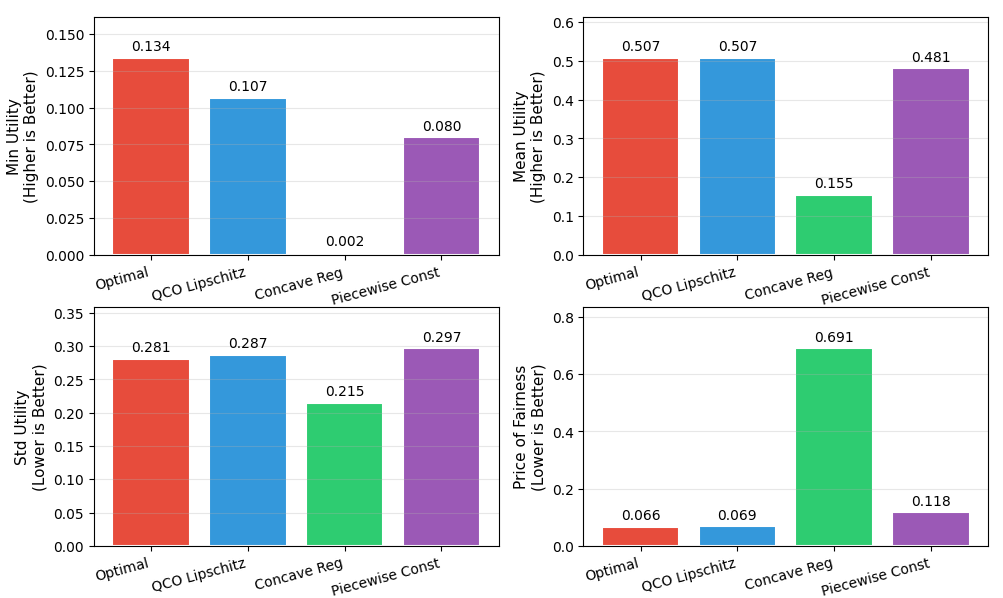}
    \caption{Utility Comparison Across Individuals}
    \label{fig:Maxmin_Fairness_Comparison}
\end{figure}

\subsubsection{Computational Efficiency}\label{subsubsec:maxmin-efficiency}
{\color{black}We compare the running time of our binary search algorithm (Algorithm~\ref{algo:binary}) against the level function method (Algorithm~\ref{algo:level_function}) for solving the robust max-min utility problem as the number of sampled points $J$ increases, see Table~\ref{table:scalability-samples-maxmin}. Binary search is consistently faster and the gap widens as $J$ increases. For $J=1024$ sampled points, binary search completes in $0.46$ seconds while the level function method requires $30.5$ seconds. These experiments show that our algorithm is computationally better suited for large-scale problems.}

\begin{table}[ht]
\begin{centering}

\begin{tabular}{ccccccccccc}
\hline
 & &  \multicolumn{7}{c}{Number of Sampled Points $J$}\tabularnewline
 \cline{3-9}

 & Method &  16 &  32 &  64 &  128 & 256 & 512 & 1024
\tabularnewline
\hline 

& Binary Search  & 0.0294 & 0.0382 & 0.0483 & 0.0786 & 0.1469 & 0.2523 & 0.4617 \\
& Level Function  & 0.1689 & 0.3183 & 0.7678 & 2.3155 & 7.0634 & 14.1898 & 30.5062 \\

\hline 
\end{tabular}
\par\end{centering}
\caption{Running time of the binary search algorithm compared to the level function method with increasing sampled points and fixed dimension $N=2$ (in seconds).}
\label{table:scalability-samples-maxmin}
\end{table}

\section{Conclusion}\label{sec:conclusion}
This paper has put forward a new RO framework for solving quasiconcave maximization problems with an ambiguous objective. This framework covers several existing convex RO models as special cases, such as robust linear programming with an ambiguous objective. Despite this greater generality, we can still solve our RO problem efficiently based on its special problem structure. In particular, we can fully characterize the upper level sets of the worst-case objective function.
We then only need to solve a logarithmic number of convex optimization problems based on these upper level sets to solve our RO problem. In this view, solving quasiconcave RO problems is not much harder than solving some ordinary convex RO problems. Our numerical experiments on a Cobb-Douglas production efficiency problem {\color{black} and a max-min utility problem} further support the value of our new RO framework. Approximation with piecewise linear quasiconcave functions (our method) outperformed piecewise constant quasiconcave and concave approximation. Furthermore, our binary search algorithm can solve our RO problem much more efficiently compared to general purpose quasiconcave maximization algorithms (e.g., the level function method).

There are several directions for future research. First, our RO model here is based on the quasiconcave envelope of a data sample subject to additional functional properties. We would like to incorporate other forms of ambiguity sets, e.g., those based on a ball, into robust quasiconcave maximization.
Second, we would like to extend our method to  quasiconcave maximization with additional stochastic uncertainty.
Finally, we would like to address maximizing quasiconcave functions in the online setting, where we seek low regret algorithms based only on function evaluations.

\bibliography{References.bib}

\newpage
\appendix
\numberwithin{equation}{section}
\setcounter{equation}{0}
\setcounter{page}{1}
\small

\begin{center}
    \interlinepenalty=10000
    \Large Appendix for ``Robust Data-Driven Quasiconcave Optimization''
\end{center}

\section{Alternative Representation}
\label{sec:alternative}
This section gives an alternative representation of our worst-case objective function which highlights the role of a set of convex performance functions and targets for all levels of the worst-case objective.
This result shows that $\psi_{{\cal U}}$ can be represented by a set of $J$ convex functions, and a continuum of targets for all levels. We explicitly determine this set of convex functions (which are closely related to the upper level sets of $\psi_{{\cal U}}$) as well as the targets. This result is based on the choice model in \cite{brown2012aspirational} and its representation in terms of a family of convex risk measures and targets.


We first derive the representation of $\psi_{{\cal U}}$.
We can represent $\psi_{{\cal U}}$ in terms of a set of convex functions (one for every input $\theta_j$ for $j \in [J]$) and targets (one for every level $\upsilon \leq \upsilon_{\max}$).
First we define constants
$
c_j \triangleq -\inf \left\{m\in \mathbb{R} \mid m\, {\vec 1}_{N} \geq \sum_{\theta\in \mathcal{D}_j} \tilde{\theta} \cdot p_\theta,\,
\sum_{\theta\in \mathcal{D}_j} p_\theta =1,\, p \in \mathbb{R}^j_{\geq 0}\right\},\,j \in [J],
$
and upper level sets
$
{\cal A}_j \triangleq \left\{x\in {\cal X} \mid x \geq \sum_{\theta\in \mathcal{D}_j} \tilde{\theta} \cdot p_\theta + c_j,\,\sum_{\theta\in \mathcal{D}_j} p_\theta =1,\, p \in \mathbb{R}^j_{\geq 0}\right\},\,j \in [J],
$
corresponding to each $\mathcal{D}_j$ (note these correspond to the characterization of the upper level sets in Theorem~\ref{thm:upper_level}).
Then, we define convex functions $\mu_j : \mathbb{R}^N \rightarrow \mathbb{R}$ via:
$
\mu_j(x) \triangleq \inf_{m \in \mathbb{R}} \left\{m \mid x + m\,{\vec 1}_{N} \in {\cal A}_j \right\},\,\forall x \in {\cal X},
$
for all $j \in [J]$. We can interpret $\mu_j(x)$ as the minimum amount which must be added to every component of $x$ to put it in the upper level set ${\cal A}_j$, so smaller values of $\mu_j(x)$ mean $x$ is closer to/deeper into ${\cal A}_j$. The value $\mu_j(x)$ is the cost of putting $x$ into ${\cal A}_j$.

Next we establish the properties of $\{\mu_j\}_{j \in [J]}$.
We say $\mu_j$ is translation invariant (with respect to the direction ${\vec 1}_N$) if $\mu_{j}(x + \beta\,{\vec 1}_{N}) = \mu_{j}(x) - \beta$ for all $\beta \in \mathbb{R}$.
We say $\mu_j$ is normalized if $\mu_j(0) = 0$ for all $j \in [J]$.
We require the following technical condition so that $\mu_j$ are normalized.

\begin{assumption}
\label{assu:zero}
We have $0 \in {\cal X}$.
\end{assumption}
\noindent
Assumption~\ref{assu:zero} can be met without loss of generality by translation of the data points in $\Theta$.

\begin{proposition}
\label{prop:aspiration risk measure}
Suppose Assumption~\ref{assu:zero} holds.

(i) $\{\mu_j\}_{j \in [J]}$ are monotone, convex, translation invariant, and normalized.

(ii) $\mu_j \geq \mu_{j+1}$ for all $j \in [J-1]$.
\end{proposition}
\begin{proof}
    (i) First, note that all of the sets ${\cal A}_j$ are monotone by construction. If $x \in {\cal A}_j$ and $y \geq x$, then $y \in {\cal A}_j$.
Next, choose $x \leq y$, then $y + m {\vec 1}_{N} \geq x + m {\vec 1}_{N}$ for all $m \in \mathbb{R}$. It follows that $\mu_{j}(x) \geq \mu_{j}(y)$.

All ${\cal A}_j$ are convex because they are polyhedra. Pick $x_1, x_2 \in {\cal X}$, $\epsilon > 0$, and $\lambda \in [0,1]$. For $i = 1, 2$, let $m_i$ satisfy $x_i + m_i {\vec 1}_{N}$ and $m_i \leq \mu_j(x_i) + \epsilon$. Then $x_i + m_i {\vec 1}_{N} \in {\cal A}_j$ for $i=1,2$, and so
$
\lambda (x_1 + m_1 {\vec 1}_{N}) + (1-\lambda) (x_2 + m_2 {\vec 1}_{N}) = \lambda x_1 + (1-\lambda) x_2 + (\lambda m_1 + (1-\lambda) m_2) {\vec 1}_{N} \in {\cal A}_j,
$
by convexity of ${\cal A}_j$. It follows that
$
\mu_j(\lambda x_1 + (1-\lambda) x_2) \leq \lambda m_1 + (1-\lambda) m_2 \leq \lambda \mu_j(x_1) + (1-\lambda) \mu_j(x_2) + \epsilon.
$
Since $\epsilon$ was arbitrary, we conclude that $\mu_j$ is convex.

To check translation invariance, for any $\beta \in \mathbb{R}$ we have:
\begin{align*}
\mu_{j}(x + \beta\,{\vec 1}_{N}) = & \inf_{m\in \mathbb{R}} \left\{m \mid x + (m + \beta) {\vec 1}_{N} \in \mathcal{A}_j \right\}\\
= & \inf_{y\in \mathbb{R}} \left\{y - \beta \mid x + y\, {\vec 1}_{N} \in \mathcal{A}_j \right\}\\
= & \inf_{m\in \mathbb{R}} \left\{m \mid x + m\, {\vec 1}_{N} \in \mathcal{A}_j \right\} - \beta,
\end{align*}
which shows that $\mu_j$ is translation invariant.

To check normalization, let $x=0$ by Assumption~\ref{assu:zero} so $x + m\,{\vec 1}_{N} = m\,{\vec 1}_{N}$. By the definition of $c_j$, for any $m > 0$, we have:
$
x + m\,{\vec 1}_{N} - c_j = m\,{\vec 1}_{N} - c_j \geq \sum_{\theta\in \mathcal{D}_j} \tilde{\theta} \cdot p_\theta
$
for some $p$. It then follows that $x + m\,{\vec 1}_{N} \in \mathcal{A}_j$ by definition of $\mu_j$. Now, for any $m <0 $, there is no $p$ such that $x + m\,{\vec 1}_{N} - c_j \geq \sum_{\theta\in \mathcal{D}_j} \tilde{\theta} \cdot p_\theta$, which means that $x + m\,{\vec 1}_{N} \notin \mathcal{A}_j$. Hence, we conclude $\mu_j(0) = 0$.

(ii) By construction, ${\cal A}_j \subset {\cal A}_{j+1}$ since ${\cal D}_j \subset {\cal D}_{j+1}$ for all $j\in[J-1]$. If $x + m\, {\vec 1}_{N} \in \mathcal{A}_j$, then $x + m\, {\vec 1}_{N} \in \mathcal{A}_{j+1}$. Consequently, $\mu_{j+1}(x) \leq \mu_j(x)$.
\end{proof}
Proposition~\ref{prop:aspiration risk measure} has close connections with the theory of risk measures.
In the language of risk measures, $\mathcal{A}_j$ are ``acceptance sets'' which describe a set of satisfactory positions. Notice that they are convex sets, and also monotone so if $x \in \mathcal{A}_j$ is acceptable then all $y \geq x$ are as well. We have $\mathcal{A}_j \subset \mathcal{A}_{j+1}$ so the acceptance sets become more stringent as $j$ decreases.
Corresponding to each acceptance set $\mathcal{A}_j$, $\mu_j$ is called a convex risk measure which returns the minimum amount of cash $m$ needed to make $x + m$ acceptable. Then, we have the relationship $\mathcal{A}_j = \{x : \mu_j(x) \leq 0\}$.

Now we define the target function $\tau(\upsilon) \triangleq \upsilon/L-c_{\kappa(\upsilon)}$ for all levels $\upsilon \leq \upsilon_{\max}$. We specify that $\tau(\upsilon)$ is the target for all $N$ components of $x$ (so $\tau(\upsilon) {\vec 1}_N$ appears in the following expression to compare to $x$).

\begin{theorem}
\label{thm:aspiration}
Suppose Assumption~\ref{assu:zero} holds.

(i) The target function $\tau(\upsilon)=\upsilon/L-c_{\kappa(\upsilon)}$ is non-decreasing in $\upsilon$.

(ii) We have
\begin{equation}
\label{eq:aspirational}
    \psi_{{\cal U}}(x) = \sup_{\upsilon \leq \upsilon_{\max}}\left\{ \upsilon \mid \mu_{\kappa(\upsilon)}(x - \tau(\upsilon) {\vec 1}_{N})\leq 0\right\},\, \forall x \in {\cal X}.
\end{equation}
\end{theorem}
\begin{proof}
    First, we have $\psi_{{\cal U}}(x) = \sup\{\upsilon \leq \upsilon_{\max} : x \in \mathcal{A}(\psi_{{\cal U}}, \upsilon)\}$ for all $x \in {\cal X}$ by Proposition~\ref{prop:level}.
By Theorem~\ref{thm:upper_level}, we can write $\mathcal{A}(\psi_{{\cal U}}, \upsilon)$ explicitly as:
\[\mathcal{A}(\psi_{{\cal U}}, \upsilon) = \left\{ x \mid x \geq \sum_{j=1}^{\kappa(\upsilon)} \tilde{\theta} \cdot p_\theta + \upsilon/L ,\,
\sum_{j=1}^{\kappa(\upsilon)} p_\theta =1,\, p \in \mathbb{R}^j_{\geq 0}\right\}.\]
It then follows that:
\begin{align*}
\psi_{{\cal U}}(x) &= \sup\left\{\upsilon \leq \upsilon_{\max} \mid  x \geq \sum_{j=1}^{\kappa(\upsilon)} \tilde{\theta} \cdot p_\theta + \upsilon/L ,\,
\sum_{j=1}^{\kappa(\upsilon)} p_\theta =1,\, p \geq 0\right\}\\
& = \sup\left\{\upsilon \leq \upsilon_{\max} \mid  x - \upsilon/L + c_{\kappa(\upsilon)} \in \mathcal{A}_{\kappa(\upsilon)}\right\}\\
& = \sup\left\{\upsilon \leq \upsilon_{\max} \mid  \mu_{\kappa(\upsilon)}(x -\upsilon/L + c_{\kappa(\upsilon)})\leq 0\right\}\\
& = \sup\left\{\upsilon \leq \upsilon_{\max} \mid  \mu_{\kappa(\upsilon)}(x -\tau(\upsilon))\leq 0\right\}.
\end{align*}
In the above display: the second equality follows from the definition of $\mu_j$ where $j = \kappa(\upsilon)$; the third equality follows from the fact that $\mu_j(x)\leq 0$ for $x \in {\cal X}$ if and only if $x \in \mathcal{A}_j$; and the last equality follows from the definition of the target function $\tau(\upsilon)$.
\end{proof}

\noindent
Eq.~\eqref{eq:aspirational} can be viewed as a decomposition of $\psi_{{\cal U}}$ into a set of convex cost functions $\{\mu_j\}_{j \in [J]}$ and a target function $\tau(\upsilon)$. Note that there is a convex cost function for every input $j \in [J]$ in the original data sample.


We can alternatively choose upper level sets and targets to construct an evaluation function $\psi \in \mathcal{F}_{\text{QCo}}$.
We can generalize the more specific Theorem~\ref{thm:aspiration} with the following ingredients:
\begin{itemize}
    \item An index function $\kappa : \mathbb{R} \rightarrow [L]$ that is decreasing for $L \geq 1$.
    \item A collection $\{\mu_l\}$ of monotone, convex, translation invariant, and normalized functions. Furthermore, $\{\mu_l\}$ satisfy $\mu_l \geq \mu_{l+1}$ for all $l \in [L-1]$.
    \item A target function $\vec{\tau} : \mathbb{R} \rightarrow \mathbb{R}^N$ that is component-wise non-decreasing in $\upsilon$, where $\vec{\tau}(\upsilon) = (\tau_1(\upsilon), \ldots, \tau_N(\upsilon))$ and $\tau_n(\upsilon)$ is the target for component $n$ at level $\upsilon$.
\end{itemize}
This setup allows more flexibility in designing a valuation function $\psi : {\cal X} \rightarrow \mathbb{R}$ defined by:
\begin{equation}
\label{eq:construct_targets}
    \psi(x) = \sup\{\upsilon : \mu_{\kappa(\upsilon)}(x - \vec{\tau}(\upsilon)) \leq 0\},\, x \in {\cal X}.
\end{equation}
Eq.~\eqref{eq:construct_targets} can be interpreted as the largest valuation $\upsilon$ such that $x$ is acceptable relative to the target $\vec{\tau}(\upsilon)$.

\begin{theorem}
\label{thm:construct_targets}
Let $\psi$ be defined by Eq.~\eqref{eq:construct_targets}. Then $\psi \in \mathcal{F}_{\text{QCo}}$.
\end{theorem}
\begin{proof}
    First we verify that $\psi$ is monotone. Choose $x \in {\cal X}$, $\epsilon > 0$, and $\upsilon$ such that $\mu_{\kappa(\upsilon)}(x - \vec{\tau}(\upsilon)) \leq 0$ and $\upsilon \geq \psi(x) - \epsilon$. For any $y \geq x$, since all $\{\mu_l\}$ are monotone we have $\mu_{\kappa(\upsilon)}(y - \vec{\tau}(\upsilon)) \leq 0$. It follows that $\psi(y) \geq \upsilon \geq \psi(x) - \epsilon$. Since $\epsilon$ was arbitrary, we conclude $\psi(y) \geq \psi(x)$.

Next we verify that $\psi$ is quasiconcave. Choose $x_1, x_2 \in {\cal X}$. Without loss of generality, suppose $\psi(x_1) \leq \psi(x_2)$.
Then choose $\epsilon > 0$, and $\upsilon_1, \upsilon_2$ such that $\upsilon_1 \leq \upsilon_2$ and $\upsilon_i \geq \psi(x_i) - \epsilon$ for $i=1,2$. We have $\mu_{\kappa(\upsilon_i)}(x_i - \vec{\tau}(\upsilon_i)) \leq 0$ for $i=1,2$ by definition of $\psi$. In addition, we have:
$
\mu_{\kappa(\upsilon_1)}(x_2 - \vec{\tau}(\upsilon_1)) \leq \mu_{\kappa(\upsilon_1)}(x_2 - \vec{\tau}(\upsilon_2)) \leq \mu_{\kappa(\upsilon_2)}(x_2 - \vec{\tau}(\upsilon_2)) \leq 0,
$
where the first inequality follows by monotonicity of $\mu_{\kappa(\upsilon_1)}$ and the second inequality follows by monotonicity of $\mu_l$ in $l$.
By convexity of $\mu_{\kappa(\upsilon_1)}$, we have:
$
\mu_{\kappa(\upsilon_1)}(\lambda x_1 + (1-\lambda) x_2 - \vec{\tau}(\upsilon)) \leq 0,\, \forall \lambda \in [0,1].
$
Then, we have
$
\psi(\lambda x_1 + (1-\lambda) x_2) \geq \upsilon_1 \geq \min\{\psi(x_1), \psi(x_2)\} - \epsilon.
$
Since $\epsilon$ was arbitrary, we have $\psi(\lambda x_1 + (1-\lambda) x_2) \geq \min\{\psi(x_1), \psi(x_2)\}$.
\end{proof}
The target function for our worst-case evaluation function $\psi_{{\cal U}}$ in Eq.~\eqref{eq:aspirational} is $\vec{\tau}(\upsilon) = \tau(\upsilon) {\vec 1}_{N}$ where all the components are identical, and the convex loss functions $\{\mu_j\}_{j \in [J]}$ are based on the upper level sets $\{{\cal A}_j\}_{j \in [J]}$ corresponding to our data sample. Eq.~\eqref{eq:construct_targets} generalizes this construction by giving more flexibility in the choice of convex performance functions and targets.

\section{Permutation Invariance}
\label{sec:invariance}

In this section, we incorporate the property of permutation invariance into our robust choice model.
This property naturally emerges in many applications.
For instance, we may be evaluating the utility of a group of identical customers, so interchanging their order should not change the overall evaluation.
Permutation invariance also corresponds to the property of law invariance of choice functions in the decision theory literature.

To define a set of permutation invariant functions in $\mathfrak{F}$, we first decompose the components of $x \in \mathbb{R}^N$ into groups of equal size.
Fix $M \geq 1$ to be the number of groups, and suppose $N = M K$ where $K \geq 1$ is the size of each group (so $M$ evenly divides $N$).
Then let the set $\Omega = \{\omega_1, \ldots, \omega_M\}$ index groups, so we can write $x = (x(\omega_1),\ldots,x(\omega_M))$ where each $x(\omega_m) = (x_{K (m-1) + 1}, \ldots, x_{K m}) \in \mathbb{R}^K$ for $m \in [M]$.
For example, in multi-objective stochastic optimization, $M$ is the number of problem scenarios (for optimization under uncertainty) and $K$ is the number of problem attributes.
This setup gives us flexibility in the implementation of permutation invariance, depending on the problem at hand.

Now let $\Sigma$ denote the set of all permutations $[M] \rightarrow [M]$, and $\sigma \in \Sigma$ denote a specific permutation. Then we define $\sigma(x) = (x(\omega_{\sigma(1)}), \ldots, x(\omega_{\sigma(M)}))$ to be a permutation of $x$ according to $\sigma$. 
It changes the order of the groups defined above (but does not change the order of elements within each group).

\begin{definition}
(a) A function $f : {\cal X} \rightarrow \mathbb{R}$ is permutation invariant if $f(x) = f(\sigma(x))$ for all $\sigma \in \Sigma$.

(b) We let $\mathcal{F}_{\text{QCo}}^{\dagger}$ denote the set of all permutation invariant functions in $\mathcal{F}_{\text{QCo}}$.
\end{definition}

We let
$
{\cal U}^{\dagger} = {\cal U}^{\dagger}(\mathsf{D}, L) = \{f \in \mathcal{F}_{\text{QCo}}^{\dagger} \cap {\cal F}_{\text{Lip}}(L) : f(\theta) \geq \hat{v}(\theta),\, \forall \theta \in \Theta\},
$
denote our ambiguity set for the permutation invariant case. In particular, the permutation invariant ambiguity set ${\cal U}^{\dagger} \subset {\cal U}$ is a strict subset of the original ambiguity set.

The corresponding worst-case evaluation function is $\psi_{{\cal U}^{\dagger}}$, and our RO problem becomes:
$$
\mathscr{P}({\cal U}^{\dagger}) : \max_{z \in {\cal Z}} \psi_{{\cal U}^{\dagger}}(G(z)).
$$
We will solve $\mathscr{P}({\cal U}^{\dagger})$ using the same approach we did for $\mathscr{P}({\cal U})$, computing the upper level sets and using binary search.
Our first step is to obtain the upper level sets of $\psi_{{\cal U}^{\dagger}}$.
We introduce the following additional notation:
\begin{itemize}
    \item $\vec x_k = (x_k(w_m))^M_{m=1}$ is the vector of components corresponding specifically to component $k \in [K]$ of $x$ across all $M$ groups. Then, $\sigma(\vec x_k)=(x_k(\omega_{\sigma(m)}))_{m=1}^M$ is a permutation of $\vec x_k$ which just permutes the order of the groups for fixed component $k$.
    \item $\theta_k$ is the vector of components corresponding to component $k \in [K]$ of input $\theta \in \mathbb{R}^N$ across all groups;
    \item $s_k\in \mathbb{R}^{M}$ is the subgradient of $f$ at $\theta$ corresponding to component $k \in [K]$ across all groups;
    \item $y(\theta)\in \mathbb{R}^M$ and $w(\theta) \in \mathbb{R}^M$ are auxiliary decision variables for every $\theta\in \mathcal{D}_{j}$. Let $y = \{y(\theta')\}_{\theta'\in \mathcal{D}_{j}}$ and $w = \{w(\theta')\}_{\theta'\in \mathcal{D}_{j}}$.
\end{itemize}
We define the following LP, which is used to check the upper level sets of $\psi_{{\cal U}^{\dagger}}$:
\begin{subequations}
\begin{align*}
\mathscr{P}_{\textsf{LP}}^{\dagger}(\theta;\,\mathcal{D}_{j}) : \min_{\upsilon, \xi, y, w} \quad & \upsilon\\
{\rm s.t.} \quad &  \langle {\vec 1}_{M},  y(\theta') \rangle+ \langle {\vec 1}_{M},  w(\theta') \rangle -\langle \xi, \theta \rangle+ \upsilon-\hat{v}(\theta')\geq 0, \quad \forall \theta'\in \mathcal{D}_{j},\\
& \sum^K_{k=1}\theta'_{k} s_k^{\top}-y(\theta') {\vec 1}_{M}^{\top}-{\vec 1}_{M} w(\theta')^{\top} \geq 0, \quad \forall \theta'\in \mathcal{D}_{j},\\
& \xi \geq 0,\left\|\xi\right\|_{1} \leq L.
\end{align*}
\end{subequations}
In particular, $\mathscr{P}_{\textsf{LP}}^{\dagger}(\theta;\,\mathcal{D}_{j})$ has a polynomial number of constraints even though there is an exponential number of permutations. This feature is essential for our binary search algorithm to be tractable.

\begin{proposition}\label{prop:interpolation-value-check-law}
Fix $x \in \mathbb{R}^N$ and let $\upsilon = \psi_{{\cal U}^{\dagger}}(x)$. Then $\psi_{{\cal U}^{\dagger}}(x) = \min\{\hat{v}(\theta_j), \textsf{val}(\mathscr{P}_{LP}^{\dagger}(x; \mathcal{D}_j))\}$ for $j = \kappa(\upsilon)$.
\end{proposition}
\begin{proof}
We first introduce the disjunctive programming problem
\begin{subequations}
\label{prob:interpolation_prm}
\begin{align}
\mathscr{P}^{\dagger}(x; \mathcal{D}_j) : \min_{\upsilon \in \mathbb{R}, \xi \in \mathbb{R}^N} \quad & \upsilon\\
{\rm s.t.} \quad & \upsilon + \max\{\langle \xi, \sigma(\theta) - x \rangle, 0\} \geq \hat{v}(\theta),\, \forall \theta \in \mathcal{D}_j,\, \sigma \in \Sigma,\label{prob:interpolation_prm-2}\\
& \xi \geq 0,\, \|\xi\|_1 \leq L.
\end{align}
\end{subequations}
By the same reasoning as Proposition~\ref{prop:disjunctive}, we have that $\psi_{{\cal U}^{\dagger}}(x) = \textsf{val}(\mathscr{P}^{\dagger}(x; \mathcal{D}_j))$.
Next let $\Sigma(\mathcal{D}_j) \triangleq \{ \sigma(\theta) : \theta \in \mathcal{D}_j, \sigma \in \Sigma \}$ and introduce the LP:
\begin{subequations}
\label{prob:interpolation_prm_LP}
\begin{align}
\mathscr{P}_{LP}(x; \Sigma(\mathcal{D}_j)) : \min_{\upsilon \in \mathbb{R}, \xi \in \mathbb{R}^N} \quad & \upsilon\\
{\rm s.t.} \quad & \upsilon + \langle \xi, \sigma(\theta) - x \rangle \geq \hat{v}(\theta),\, \forall \theta \in \mathcal{D}_j,\, \sigma \in \Sigma,\label{prob:interpolation_prm_LP-2}\\
& \xi \geq 0,\, \|\xi\|_1 \leq L.
\end{align}
\end{subequations}
By the same reasoning as Proposition~\ref{prop:interpolation-value-check}, we have $\psi_{{\cal U}^{\dagger}}(x) = \min\{ \hat{v}(\theta_j), \textsf{val}(\mathscr{P}_{LP}(x; \Sigma(\mathcal{D}_j))) \}$.

We note that Eq.~\eqref{prob:interpolation_prm_LP-2} has an exponential number of constraints due to the index $\sigma \in \Sigma$. We obtain a reduction of the cut generation problem as follows.
Constraint~\eqref{prob:interpolation_prm_LP-2} is equivalent to:
\begin{equation}\label{eq:permutation-invariant-exp-constraint}
    \min_{\sigma\in \Sigma} \langle \xi, \sigma (\theta')\rangle -\langle \xi,\theta \rangle+\upsilon-\hat{v}(\theta')\geq 0, \quad \forall \theta'\in \mathcal{D}_{t}. 
\end{equation}
We will reduce the optimization problem $\min_{\sigma\in \Sigma} \langle \xi, \sigma (\theta')\rangle$ in Eq. \eqref{eq:permutation-invariant-exp-constraint}. Recall that $\xi_k$ is the subgradient of $f$ at $\theta$ corresponding to attribute $k \in [K]$. The optimal value of $\min_{\sigma\in \Sigma} \langle \xi, \sigma (\theta')\rangle$ in Eq. \eqref{eq:permutation-invariant-exp-constraint} is equal to the optimal value of:
\begin{subequations}\label{eq:assignment_multi}
\begin{eqnarray}
\min_{Q \in \mathbb{R}^{M \times M}} \, && \sum^K_{k=1}\xi_k^{\top }Q\, \theta'_{k} \\
{\rm s.t.} \, && Q^\top {\vec 1}_M={\vec 1}_M,\\
&& Q {\vec 1}_M = {\vec 1}_M,\\
&& Q_{m, l} \in\{0,1\}, \quad \forall l,\,m \in [M],
\end{eqnarray}
\end{subequations}
which is a linear assignment problem. Here $Q$ is the permutation matrix corresponding to $\sigma$ so that $Q\, \theta'_k = \sigma(\theta'_k)$ for all $k \in [K]$ (the permutation must be the same for all attributes, hence we only have a single permutation matrix $Q$). Problem \eqref{eq:assignment_multi} can be solved exactly by relaxing the binary constraints $Q_{m, l} \in\{0,1\}$ to $0 \leq Q_{m, l} \leq 1$ for all $l,\,m \in [M]$. Strong duality holds for the relaxed problem, and the optimal value of the relaxed problem is equal to:
\begin{eqnarray*}
\max_{w,y \in \mathbb R^{M}} \, && \langle {\vec 1}_M,  w \rangle+ \langle {\vec 1}_M, y \rangle \\
{\rm s.t.} \, && \sum^K_{k=1}\theta'_{k} \xi_k^{\top}-w {\vec 1}_M^{\top}-{\vec 1}_M y^{\top} \geq 0.
\end{eqnarray*}
It follows that constraint \eqref{eq:permutation-invariant-exp-constraint} is satisfied if and only if there exists $w$ and $y$ such that:
\begin{eqnarray*}
&& \langle {\vec 1}_M,  w \rangle+ \langle {\vec 1}_M, y \rangle -\langle \xi,\theta \rangle+\upsilon- \hat{v}(\theta')\geq 0,\\
&& \sum^K_{k=1}\theta'_{k} \xi_k^{\top}-w {\vec 1}_M^{\top}-{\vec 1}_M y^{\top} \geq 0,
\end{eqnarray*}
which gives the desired form of $\mathscr{P}^{\dagger}_{LP}(x;\,\mathcal{D}_j)$.
\end{proof}
The dual of $\mathscr{P}_{\textsf{LP}}^{\dagger}(x;\mathcal{D}_{j})$ leads to an explicit characterization of the upper level sets of $\psi_{{\cal U}^{\dagger}}$ in the following theorem.
\begin{theorem}\label{thm:permutation-upper}
The upper level set of $\psi_{{\cal U}^{\dagger}}$ at level $\upsilon \in (-\infty, \upsilon_{\max}]$ satisfies:
\begin{equation*}
\mathcal{A}(\psi_{{\cal U}^{\dagger}},\upsilon)=\left\{\begin{array}{l|l}
x \in {\cal X} & \begin{array}{l}
\sum_{\theta\in \mathcal{D}_{\kappa(\upsilon)}} \hat{v}(\theta) \cdot p_\theta-L\,q\geq \upsilon,\\
\sum_{\theta\in \mathcal{D}_{\kappa(\upsilon)}} \rho_\theta^{\top}\cdot \theta_k -\vec x_k\leq q, \quad \forall k \in [K]\\
\sum_{\theta\in \mathcal{D}_{\kappa(\upsilon)}} p_\theta =1, \,p \in \mathbb{R}^j_{\geq 0},\, q\geq 0,\\
\vec{1}_M^\top \rho_\theta  = p_\theta \vec{1}_M^\top,\,\rho_\theta \vec{1}_M  = p_\theta \vec{1}_M,\, \rho_\theta \in \mathbb{R}^M_{\geq 0}\times \mathbb{R}^M_{\geq 0}, \, \forall \theta\in\mathcal{D}_{\kappa(\upsilon)}
\end{array}
\end{array}\right\}.
\end{equation*}
\end{theorem}
\begin{proof}
First fix $j = \kappa(\upsilon)$.
The dual of $\mathscr P_{\textsf{LP}}^{\dagger}(x;\,\mathcal{D}_{j})$ is as follows. First define variables $p\in \mathbb{R}^j_{\geq 0}$, $q\in \mathbb{R}$, and $\{\rho_\theta\}_{\theta\in\mathcal{D}_{j}}$ where $\rho_\theta\in \mathbb{R}^M_{\geq 0}\times \mathbb{R}^M_{\geq 0}$. The dual to $\mathscr P_{\textsf{LP}}^{\dagger}(x;\,\mathcal{D}_{j})$ is then:
\begin{subequations}
\begin{align}
\mathsf{D}^{\dagger}(x;\mathcal{D}_{j}) : \max_{p,q,\{\rho_\theta\}_{\theta\in\mathcal{D}_{j}} } \quad & \sum_{\theta\in \mathcal{D}_{j}} \hat{v}(\theta) \cdot p_\theta-L\,q\label{dual-law-1}\\
{\rm s.t.} \quad & \sum_{\theta\in \mathcal{D}_{j}} \rho_\theta^{\top} \theta_k -\vec x_k\leq q, \quad \forall k \in [K],&\label{dual-law-2}\\
& \sum_{\theta\in \mathcal{D}_{j}} p_\theta =1, \,p \in \mathbb{R}^j_{\geq 0},\, q\geq 0, &\label{dual-law-3}\\
& \vec{1}_M^\top \rho_\theta = p_\theta \vec{1}_M^\top,\,\rho_\theta \vec{1}_M  = p_\theta \vec{1}_M,\, \rho_\theta \in \mathbb{R}^M_{\geq 0}\times \mathbb{R}^M_{\geq 0}, \quad \forall \theta\in\mathcal{D}_{j}. &\label{dual-law-4}
\end{align}
\end{subequations}
We interpret constraint \eqref{dual-law-2} in the component-wise sense.

Problem $\mathscr P_{\textsf{LP}}^{\dagger}(x;\,\mathcal{D}_{j})$ is always feasible and its optimal value is lower bounded, so we have $\textsf{val}(\mathscr P_{\textsf{LP}}^{\dagger}(x;\,\mathcal{D}_{j})) = \textsf{val}(\mathsf{D}^{\dagger}(x;\mathcal{D}_{j}))$ for all $x \in {\cal X}$ by strong duality.
The conclusion then follows from Proposition \ref{prop:interpolation-value-check-law}.
\end{proof}

Theorem~\ref{thm:permutation-upper} shows that the permutation invariant upper level sets also have a polyhedral structure (with a polynomial number of variables and constraints).
With the upper level sets in hand, our strategy for solving $\mathscr{P}({\cal U}^{\dagger})$ is analogous to the base case.
First note $\mathscr{P}({\cal U}^{\dagger})$ is equivalent to
\begin{equation}
\label{prob:robust_permutation}
\max_{z \in \mathcal Z,\, \upsilon \in \mathbb{R}}\{ \upsilon : G(z) \in \mathcal{A}(\psi_{{\cal U}^{\dagger}}, \upsilon) \},
\end{equation}
where we want to find the largest value $\upsilon$ such that there exists $z \in {\cal Z}$ with $G(z) \in \mathcal{A}(\psi_{{\cal U}^{\dagger}}, \upsilon)$.

Fix $j \in [J]$, and introduce decision variables $\upsilon\in \mathbb{R}$, $z\in {\cal Z}$, $p\in \mathbb{R}^j$, $q\in \mathbb{R}$, and $\{\rho_\theta\}_{\theta\in\mathcal{D}_{j}}$ where $\rho_\theta\in \mathbb{R}^M\times \mathbb{R}^M$.
We define $\vec g_k(z) = (g_k(z, w_m))^M_{m=1}$ to be the vector of components of $G(z)$ corresponding specifically to component $k \in [K]$ of $x(\omega_m)$ over all groups $m \in [M]$.
Consider the following problem similar to $\mathscr{G}(\mathcal{D}_j)$:
\begin{subequations}\label{eq:binary-optimization-law}
\begin{align}
\mathscr G^{\dagger}(\mathcal{D}_{j}) : \max_{z,p,q,\rho,\upsilon} \quad & \upsilon\\
{\rm s.t.} \quad & \sum_{\theta\in \mathcal{D}_{j}} \hat{v}(\theta) \cdot p_\theta-L\,q\geq \upsilon,\\
& \sum_{\theta\in \mathcal{D}_{j}} \rho_\theta^{\top}\cdot \theta_k -\vec g_k(z)\leq q, \quad \forall k \in [K],\\
& \sum_{\theta\in \mathcal{D}_{j}} p_\theta =1, \,p \in \mathbb{R}^j_{\geq 0},\, q\geq 0,\\
& \vec{1}_M^\top \rho_\theta = p_\theta \vec{1}_M^\top,\,\rho_\theta \vec{1}_M  = p_\theta \vec{1}_M,\, \rho_\theta \in \mathbb{R}^M_{\geq 0}\times \mathbb{R}^M_{\geq 0}, \quad \forall \theta\in\mathcal{D}_{j},
\end{align}
\end{subequations}
which is a convex optimization problem.

We show in the following proposition that we can bound the optimal value of Problem~\eqref{prob:robust_permutation} using $\mathscr G^{\dagger}(\mathcal{D}_{j})$ for $j=\kappa(\upsilon)$.
This result follows from Theorem~\ref{thm:permutation-upper} and the same argument as Proposition~\ref{prop:binary}.

\begin{proposition}\label{prop:binary_permutation}
Choose level $\upsilon \in (-\infty, \upsilon_{\max}]$ and $j = \kappa(\upsilon)$, then $\max_{z\in{\cal Z}} \psi_{{\cal U}^{\dagger}}(G(z))\geq \upsilon$ if and only if $\textsf{val} (\mathscr G^{\dagger}(\mathcal{D}_{j}))\geq \upsilon$.
\end{proposition}



We now present the details of our binary search algorithm for Problem~\eqref{prob:robust_permutation}. It mirrors Algorithm~\ref{algo:binary} except now we replace each instance of $\mathscr G(\mathcal{D}_j)$ with the modified problem $\mathscr G^{\dagger}(\mathcal{D}_{j})$.

\begin{algorithm}
\SetAlgoLined
 Initialization: data sample $\mathsf{D} = \{(\theta, \hat{v}(\theta))\}_{\theta \in \Theta}$, $j_1 = J$, $j_2 = 1$;
 
 \While{$j_1\ne j_2$}{
  Set $j: = \lfloor \frac{j_1+j_2}{2} \rfloor$, and compute $\upsilon_{j} = \textsf{val}(\mathscr G^{\dagger}(\mathcal{D}_{j}))$ with optimal solution $z^*$;
  
  \lIf{$\upsilon_{j}\leq \hat{v}(\theta_{j+1})$}{
   set $j_2:=j+1$}
   \lElse{
   set $j_1:=j$}
 }
 Set $j: = \lfloor \frac{j_1+j_2}{2} \rfloor$, and compute $\upsilon_{j} = \textsf{val}(\mathscr G^{\dagger}(\mathcal{D}_{j}))$ with optimal solution $z^*$;
 
 \Return $z^*$ and $\psi_{{\cal U}^{\dagger}}(G(z^*)) = \min\{\upsilon_{j}, \hat{v}(\theta_j)\}$.

 \caption{Binary search for $\mathscr{P}({\cal U}^{\dagger})$}\label{algo:binary_permutation}
\end{algorithm}

The complexity of Algorithm~\ref{algo:binary_permutation} for solving Problem~\eqref{prob:robust_permutation} follows the same reasoning as Theorem~\ref{thm:binary}.

\begin{theorem}\label{thm:binary_permutation}
Algorithm~\ref{algo:binary_permutation} returns an optimal solution $z^*$ of $\mathscr{P}({\cal U}^{\dagger})$, after solving at most $\log J$ instances of Problem~\eqref{eq:binary-optimization-law}.
\end{theorem}
\noindent
We note the same order of complexity in Theorem \ref{thm:binary_permutation} that we saw in Theorem~\ref{thm:binary}. Like the base case, the binary search algorithm for Problem~\eqref{prob:robust_permutation} is highly scalable in both the size of the problem instance and the level $J$ (which grows with the size of the dataset $\mathsf{D}$). The difference between these two algorithms is reflected in the form of the feasibility problems that are solved for each level.

\section{Proofs for Section~\ref{sec:preliminaries}}

\subsection{Proof of Proposition~\ref{prop:quasiconcave}}

For $z_1, z_2 \in {\cal Z}$ and $\lambda \in [0,1]$, we have:
\begin{align*}
\rho_f(\lambda z_1 + (1-\lambda) z_2) = & f(G(\lambda z_1 + (1-\lambda) z_2))\\
\geq & f(\lambda G(z_1) + (1-\lambda) G(z_2))\\
\geq & \min\{\rho_f(z_1), \rho_f(z_2)\},
\end{align*}
where the first inequality uses monotonicity of $f$ and concavity of $G$, and the second inequality uses quasiconcavity of $f$.

\section{Proofs for Section~\ref{sec:upper}}

\subsection{Proof of Proposition~\ref{prop:level}}

(i) This part is immediate by the definition of upper level set:
$$
\sup\{v\in \mathbb{R} : x \in {\cal A}(f,\upsilon)\} = 
\sup\{v\in \mathbb{R} : 
f(x)\geq v\} = f(x).
$$

(ii) We have the following equalities: 
\begin{align*}
{\cal A}(\psi_{\cal F}, \upsilon) = & \{x \in {\cal X} : \psi_{\cal F}(x) \geq \upsilon\}\\
= & \{x \in {\cal X} : \inf_{f \in {\cal F}}f(x) \geq \upsilon\}\\
= & \{x \in {\cal X} : f(x) \geq \upsilon,\, \forall f \in {\cal F}\}\\
= & \cap_{f \in {\cal F}} {\cal A}(f, \upsilon),
\end{align*}
at level $\upsilon \in \mathbb{R}$.

\subsection{Proof of Proposition~\ref{prop:disjunctive}}

We have that $\psi_{{\cal U}}(x)$ is equal to the optimal value of:
\begin{subequations}
\label{prob:interpolation_full}
\begin{align}
\mathscr{P}(x; \mathcal{D}_J) : \min_{\upsilon, \xi} \quad & \upsilon\\
{\rm s.t.} \quad & \upsilon + \max\{\langle \xi, \theta - x \rangle, 0\} \geq \hat{v}(\theta),\, \forall \theta \in \mathcal{D}_J,\label{prob:interpolation_full-2}\\
& \xi \geq 0,\, \|\xi\|_1 \leq L.
\end{align}
\end{subequations}
Let $\upsilon^* = \textsf{val}(\mathscr{P}(x; \mathcal{D}_J))$. Constraints~\eqref{prob:interpolation_full-2} are not binding for all $\theta$ with $\hat{v}(\theta) < \upsilon^*$. By definition of $\kappa(\upsilon^*)$, $\mathscr{P}(x; \mathcal{D}_{\kappa(\upsilon^*)})$ is then equivalent to $\mathscr{P}(x; \mathcal{D}_J)$ (it has the same optimal value and set of optimal solutions).

\subsection{Proof of Proposition~\ref{prop:interpolation-value-check}}

First we have $\upsilon = \textsf{val}(\mathscr{P}(x;\,\mathcal{D}_j))$ by Proposition~\ref{prop:disjunctive}.
There are two cases to check: (i) $\hat{v}(\theta_{j+1})< \upsilon<\hat{v}(\theta_j)$; and (ii) $\upsilon=\hat{v}(\theta_j)$.

For the first case, since $\upsilon < \hat{v}(\theta)$ for all $\theta \in {\cal D}_j$, we must have $\langle \xi, \theta-x \rangle > 0$ for all $\theta \in {\cal D}_j$. It follows that $\mathscr{P}(x;\,\mathcal{D}_j)$ and $\mathscr{P}_{LP}(x;\,\mathcal{D}_j)$ are equivalent, and $\upsilon = \textsf{val}(\mathscr{P}_{LP}(x;\,\mathcal{D}_j))$.
For the second case, we have $\textsf{val}(\mathscr{P}_{LP}(x;\,\mathcal{D}_j)) \geq \textsf{val}(\mathscr{P}(x;\,\mathcal{D}_j))$, and so $\textsf{val}(\mathscr{P}_{LP}(x;\,\mathcal{D}_j)) \geq \hat{v}(\theta_j)$.

\subsection{Proof of Theorem~\ref{thm:upper_level}}

Let $\mathscr{D}_{LP}(x;\mathcal{D}_j)$ denote the LP dual of $\mathscr{P}_{LP}(x;\,\mathcal{D}_j)$.
Problem $\mathscr{P}_{LP}(x;\,\mathcal{D}_j)$ is automatically feasible, e.g., take $\upsilon = \max\{\hat{v}(\theta) : \theta \in {\cal D}_j\}$ and $\xi = 0$. Furthermore, under [Lip] the optimal value of $\mathscr{P}_{LP}(x;\,\mathcal{D}_j)$ is lower bounded by
$
\min\{\hat{v}(\theta) - L \|\theta - x\|_1 : \theta \in {\cal D}_j\} > -\infty.
$
Thus, LP strong duality holds between $\mathscr{P}_{LP}(x;\,\mathcal{D}_j)$ and $\mathscr{D}_{LP}(x;\,\mathcal{D}_j)$, and we have $\textsf{val}(\mathscr{P}_{LP}(x;\,\mathcal{D}_j)) = \textsf{val}(\mathscr D_{LP}(x;\,\mathcal{D}_j))$ for all $x \in {\cal X}$.

For $\hat{v}(\theta_{j+1}) < \upsilon\leq \hat{v}(\theta_j)$ and $j \in [J]$, by Proposition \ref{prop:interpolation-value-check} {\color{black}we know that $x \in \mathcal{A}(\psi_{{\cal U}}, \upsilon)$ if and only if $\textsf{val}(\mathscr{P}_{LP}(x;\,\mathcal{D}_j))\geq \upsilon$.} By strong duality, this is equivalent to the inequality $\textsf{val}(\mathscr{D}_{LP}(x;\mathcal{D}_j)) \geq \upsilon$. However, since $\mathscr{D}_{LP}(x;\mathcal{D}_j)$ is a maximization problem, this latter inequality reduces to the feasibility problem:
\[
\left\{(p,\,q) : \sum_{\theta\in \mathcal{D}_j} \hat{v}(\theta) p_\theta-L\,q \geq \upsilon,\,\sum_{\theta\in \mathcal{D}_j} \theta \cdot p_\theta -x\leq q\cdot {\vec 1}_{N},\,\sum_{\theta\in \mathcal{D}_j} p_\theta =1,\,p \in \mathbb{R}^{j}_{\geq 0},\, q\geq 0\right\}.
\]
Since $\left( \sum_{\theta\in \mathcal{D}_j} \hat{v}(\theta) \cdot p_\theta - \upsilon \right)/L \geq q$ and $\sum_{\theta\in \mathcal{D}_j} \theta \cdot p_\theta -x\leq q \cdot {\vec 1}_{N}$, we can eliminate $q$ from the above feasibility problem to obtain:
\[
\left\{p : \sum_{\theta\in \mathcal{D}_j} \theta \cdot p_\theta -x \leq \left( \sum_{\theta\in \mathcal{D}_j} \hat{v}(\theta) \cdot p_\theta - \upsilon \right)/L\cdot {\vec 1}_{N},\,\sum_{\theta\in \mathcal{D}_j} p_\theta =1,\,p \in \mathbb{R}^{j}_{\geq 0}\right\}.
\]
We then have
\[
\sum_{\theta\in \mathcal{D}_j} \theta \cdot p_\theta -\left(\sum_{\theta\in \mathcal{D}_j} \hat{v}(\theta) \cdot p_\theta-\upsilon\right)/L \cdot {\vec 1}_{N} = \sum_{\theta\in \mathcal{D}_j}  \left(\theta-\hat{v}(\theta)/L \right)p_\theta+\upsilon/L \cdot {\vec 1}_{N},
\]
and the desired result follows from the definition of $\tilde{\theta} = \theta-\hat{v}(\theta)/L$.

{\color{black}
\section{Additional Results for Non-Monotone Functions}\label{sec:nonmonotone}
We can extend our results from Sections~\ref{sec:upper} and \ref{sec:binary} to cover non-monotone quasiconcave functions. Let $\mathcal{F}_{\text{NQCo}} \subset \mathfrak{F}$ be the set of all quasiconcave functions $f : {\cal X} \rightarrow \mathbb{R}$ (which are not necessarily monotone). We modify the definition of our ambiguity set to: 
$
{\cal U} = {\cal U}(\mathsf{D}, L) = \{f \in \mathcal{F}_{\text{NQCo}} \cap {\cal F}_{\text{Lip}}(L) : f(\theta) \geq \hat{v}(\theta),\, \forall \theta \in \Theta\}.
$
For this setting, the non-monotone counterpart of $\mathscr{P}_{LP}(x;\,\mathcal{D}_j)$ is obtained by simply dropping the constraint $\xi \geq 0$ to get:
\begin{subequations}
\label{eq:descent-nonmonotone}
\begin{align}
\mathscr{P}_{LP}(x;\,\mathcal{D}_j) : \min_{\upsilon \in \mathbb{R}, \xi \in \mathbb{R}^N} \quad & \upsilon \label{eq:descent-nonmonotone-1}\\
\textrm{s.t.} \quad & \upsilon+\left\langle \xi, \theta-x\right\rangle \geq \hat{v}(\theta),\,\forall \theta \in \mathcal{D}_j, \label{eq:descent-nonmonotone-2}\\
& \left\|\xi\right\|_{1} \leq L. \label{eq:descent-nonmonotone-3}
\end{align}
\end{subequations}
We now use the above form of $\mathscr{P}_{LP}(x;\,\mathcal{D}_j)$ to compute $\psi_{{\cal U}}(x)$ for fixed $x \in {\cal X}$.
The dual to Problem $\mathscr{P}_{LP}(x;\,\mathcal{D}_j)$ is:
\begin{subequations}
\begin{align*}
\mathscr{D}_{LP}(x;\,\mathcal{D}_j):\quad
\max_{\{p_{\theta}\}_{\theta\in\mathcal{D}_j}, t\in \mathbb{R}_{\geq 0}} \quad
& \sum_{\theta\in\mathcal{D}_j} p_{\theta}\,\hat{v}(\theta)
\;-\;
L t\\
\textrm{s.t.}\quad & -t \vec 1_N \leq \sum_{\theta\in\mathcal{D}_j} p_{\theta}\,(\theta-x) \leq t \vec 1_N,\\
& \sum_{\theta\in\mathcal{D}_j} p_{\theta} = 1,\\
& p_{\theta} \ge 0,\quad \forall \theta\in\mathcal{D}_j.
\end{align*}
\end{subequations}
Then, the constraint $\textsf{val}(\mathscr{D}_{LP}(x;\mathcal{D}_j)) \geq \upsilon$ can be reformulated as the system of inequalities:
\begin{align*}
\sum_{\theta\in\mathcal{D}_j} p_{\theta}\,\hat{v}(\theta)
\;-\;
L t \geq \upsilon,\\
-t \vec 1_N \leq \sum_{\theta\in\mathcal{D}_j} p_{\theta}\,(\theta-x) \leq t \vec 1_N, \\
\sum_{\theta\in\mathcal{D}_j} p_{\theta} = 1,\\
p_{\theta} \ge 0,\quad \forall \theta\in\mathcal{D}_j,
\end{align*}
which determine a convex (polyhedral) set.
Eliminating the dual variable $t$, we obtain the non-monotone counterpart of Problem $\mathscr G(\mathcal{D}_j)$ which is
\begin{subequations}
\begin{align*}
\mathscr G(\mathcal{D}_j) : \max_{z \in {\cal Z},\, \upsilon \in \mathbb{R},\, p \in \mathbb{R}_{\geq 0}^j} \quad & \upsilon\\
\text{s.t.} \quad & G(z) \geq \sum_{\theta\in \mathcal{D}_j} \left(\theta - (\hat{v}(\theta)/L)\vec 1_N \right) \cdot p_\theta + (\upsilon/L) {\vec 1}_{N},\\
& G(z) \leq \sum_{\theta\in \mathcal{D}_j} \left(\theta + (\hat{v}(\theta)/L)\vec 1_N\right) \cdot p_\theta - (\upsilon/L) {\vec 1}_{N},\\
& \sum_{\theta\in \mathcal{D}_j} p_\theta =1.
\end{align*}
\end{subequations}
Binary search (analogous to Algorithm~\ref{algo:binary}) can then be applied to compute the robust optimal solution using the modified $\mathscr G(\mathcal{D}_j)$.
}
\section{Proofs for Section~\ref{sec:binary}}

\subsection{Proof of Proposition~\ref{prop:binary}}

For any $\upsilon \in (-\infty, \upsilon_{\max}]$ and $j \in [J]$, we define the feasibility problem:
\begin{subequations}
\label{prob:robust-F-D-j}
\begin{align*}
(\mathscr F_{\upsilon}(\mathcal{D}_j)) \quad & \text{Find} \quad z \in {\cal Z},\, p \in \mathbb{R}_{\geq 0}^j\\
\text{s.t.} \quad & G(z) \geq \sum_{\theta\in \mathcal{D}_j} \tilde{\theta} \cdot p_\theta + (\upsilon/L) {\vec 1}_{N},\\
& \sum_{\theta\in \mathcal{D}_j} p_\theta =1,
\end{align*}
\end{subequations}
in the variables $(z, p)$.
By Theorem~\ref{thm:upper_level}, feasibility of the inequality $\psi_{{\cal U}}(G(z))\geq v$ for $z \in {\cal Z}$ is equivalent to feasibility of $\mathscr F_\upsilon(\mathcal{D}_j)$ for $j = \kappa(\upsilon)$. We then have the following chain of equivalences:
\begin{align*}
    \left\{\exists z \in {\cal Z},\, \psi_{{\cal U}}(G(z))\geq \upsilon\right\} &\iff \left\{\mathscr F_\upsilon(\mathcal{D}_j) \text{ feasible}\right\},\\
    \left\{\max_{z\in{\cal Z}} \psi_{{\cal U}}(G(z))\geq \upsilon\right\} &\iff \left\{\max \left\{\upsilon' \mid \mathscr F_{\upsilon'}(\mathcal{D}_j) \text{ feasible}\right\}\geq \upsilon\right\},\\
    \left\{\max_{z\in{\cal Z}} \psi_{{\cal U}}(G(z))\geq \upsilon\right\} &\iff \{\textsf{val}(\mathscr G(\mathcal{D}_j))\geq \upsilon\},
\end{align*}
using the fact that $\mathscr G(\mathcal{D}_j)$ is equivalent to $\max \left\{\upsilon' \mid \mathscr F_{\upsilon'}(\mathcal{D}_j) \text{ feasible}\right\}$.

\subsection{Proof of Theorem~\ref{thm:binary}}

Let $\nu$ denote the optimal value of Problem~{\color{black} \eqref{prob:robust_main}}. By Proposition \ref{prop:binary}, $\nu \geq \hat{v}(\theta_j)$ if and only if $\textsf{val} (\mathscr G(\mathcal{D}_j))\geq \hat{v}(\theta_j)$.
Hence the binary search procedure in Algorithm \ref{algo:binary} finds $\zeta \in [J]$ such that $\hat{v}(\theta_{\zeta+1}) < \nu \leq \hat{v}(\theta_{\zeta})$ within $O(\log J)$ iterations. 

To see this, note that for any $j \leq \zeta-1$, we have $\nu \leq \hat{v}(\theta_{\zeta}) \leq \hat{v}(\theta_{j+1})$. That is, by Proposition~\ref{prop:binary} we have $\textsf{val} (\mathscr G(\mathcal{D}_{j+1}))\leq \hat{v}(\theta_{j+1})$. Therefore, $\textsf{val} (\mathscr G(\mathcal{D}_{j}))\leq \hat{v}(\theta_{j+1})$ by the monotonicity of $\textsf{val} (\mathscr G(\mathcal{D}_{j}))$ in $j$ by Observation~\ref{obs:monotone}. On the other hand, for any $j \geq \zeta + 1 $, we have $\nu > \hat{v}(\theta_{\zeta+1}) \geq  \hat{v}(\theta_j)$. Therefore, by Proposition~\ref{prop:binary} we have $\textsf{val} (\mathscr G(\mathcal{D}_{j})) > \hat{v}(\theta_j) \geq \hat{v}(\theta_{j+1})$. This proves that binary search returns the correct index $\zeta$.
Let $(\tilde z^*, \tilde p^*, \tilde \upsilon^*)$ be the optimal solution of $\mathscr G(\mathcal{D}_{\zeta})$ for this $\zeta$. We want to show that $\tilde z^*$ is an optimal solution of {\color{black}Problem~\eqref{prob:robust_main}}.

First, let $\upsilon^* = \psi_{{\cal U}}(G(z^*))$ by Proposition~\ref{prop:binary} so $\nu \geq \upsilon^*$. Suppose $\nu > \upsilon^*$, then there is some $\tilde{z} \in {\cal Z}$ with $\psi_{{\cal U}}(G(\tilde z)) = \nu > \upsilon^*$. It follows that $(\tilde z, \tilde \nu, \tilde p)$ are feasible for $\mathscr G(\mathcal{D}_{\tilde \zeta})$ for $\tilde \zeta = \kappa(\nu) \leq \zeta$. However, this contradicts optimality of $(\tilde z^*, \tilde p^*, \tilde \upsilon^*)$ being the optimal solution of $\mathscr G(\mathcal{D}_{\zeta})$, since any feasible solution for $\mathscr G(\mathcal{D}_{\zeta'})$ for $\zeta' \leq \zeta$ can be extended to a feasible solution for $\mathscr G(\mathcal{D}_{\zeta})$ with the same optimal value.

\end{document}